%% file: CLT.tex
\documentclass[a4paper,11pt,reqno,noindent]{amsart}
\usepackage[centertags]{amsmath}
\usepackage{amsfonts,amssymb,amsthm} 
\usepackage[dvips]{graphicx} 
\usepackage{psfrag}
\usepackage[english]{babel}
\usepackage{newlfont}
\usepackage{color}
\usepackage[body={15cm,21.5cm},centering]{geometry} 
\usepackage{fancyhdr}
\usepackage{hyperref}
\usepackage{esint}

\newtheorem{theo}{Theorem}
\newtheorem{lemma}{Lemma}
\newtheorem{prop}{Proposition}

\theoremstyle{definition}
\newtheorem{rem}{Remark}
\newtheorem{defi}{Definition}

\numberwithin{equation}{section}
\numberwithin{lemma}{section} 
\numberwithin{defi}{section} 
\numberwithin{rem}{section} 

\newcommand \dps{\displaystyle }

\newcommand{\ho}{\mathrm{hom}}

\newcommand{\R}{\mathbb{R}}
\newcommand{\Z}{\mathbb{Z}}
\newcommand{\N}{\mathbb{N}}
\newcommand{\T}{\mathbb{T}}

\newcommand{\e}{\varepsilon}
\newcommand{\LL}{\mathcal{L}}
\newcommand{\calF}{\mathcal{F}}

\newcommand{\ee}{\mathbf{e}}
\newcommand{\ener}{\mathcal{E}}
\newcommand{\B}{\mathbb{B}}
\def\aa{\boldsymbol a}

\mathchardef\emptyset="001F

\newcommand{\dig}[1]{\mathrm{diag}\left[ #1\right]}
\newcommand{\var}[1]{\mathrm{var}\left[#1\right]}
\newcommand{\cov}[2]{\mathrm{cov}\left[#1;#2\right]}
\newcommand{\covL}[2]{\mathrm{cov}_L\left[#1;#2\right]}
\newcommand{\expec}[1]{\mathbb{E}\left[ #1 \right]}

\newcommand{\step}[1]{\noindent \textit{Step} #1.}
\newcommand{\substep}[1]{\noindent \textit{Substep} #1.}

 \newcommand{\E}[1]{\mathbb{E}\left[#1\right]}
 \newcommand{\EL}[1]{\mathbb{E}_L\left[#1\right]}

\newcommand{\expE}{\mathbb{E}}
\newcommand{\Pm}{\mathbb{P}}
\newcommand{\Rm}{\mathbb{R}}
\newcommand{\no}{\nonumber}
\newcommand{\br}{\begin{eqnarray}}
\newcommand{\er}{\end{eqnarray}}
\newcommand{\commentout}[1]{}

\title[Quantitative CLT for the effective conductance]
{A quantitative central limit theorem for the effective conductance on the discrete torus}
\author[A. Gloria]{Antoine Gloria}
\author[J. Nolen]{James Nolen}
\date{\today}
\address[Antoine Gloria]{D\'epartement de math\'ematique, Universit\'e Libre de Bruxelles, Belgium \\  and MEPHYSTO team, Inria Lille - Nord Europe, Villeneuve d'Ascq, France}
\email{agloria@ulb.ac.be}
\address[James Nolen]{Department of Mathematics, Duke University, Durham, North Carolina, USA}
\email{nolen@math.duke.edu}
\begin{document}
\maketitle

\begin{center}
\begin{minipage}{13cm}
\small{
\noindent {\bf Abstract.} We study a random conductance problem on a $d$-dimensional discrete torus of size $L > 0$. The conductances are independent, identically distributed random variables uniformly bounded from above and below by positive constants. The effective conductance $A_L$ of the network is a random variable, depending on $L$, and the main result is a quantitative central limit theorem for this quantity as $L \to \infty$. 
In terms of scalings we prove that this nonlinear nonlocal function $A_L$  essentially behaves as if it were a simple spatial average of the conductances (up to logarithmic corrections).
The main achievement of this contribution is the precise asymptotic description of the variance of $A_L$.

\vspace{10pt}

\vspace{6pt}
}

\end{minipage}
\end{center}


\section{Introduction and statement of the main result}

\input{intro}


\section{Asymptotic behavior of the rescaled variance and proof of Proposition~\ref{th:resc-var}}\label{sec:variance}

\input{variance}


\section{Normal approximation and proof of Proposition~\ref{th:normal}}\label{sec:normal}

\input{normal}


\section*{Acknowledgements}
AG acknowledges financial support from the European Research Council under
the European Community's Seventh Framework Programme (FP7/2014-2019 Grant Agreement
QUANTHOM 335410). JN acknowledges financial support from a National Science Foundation grant, DMS-1007572. The authors acknowledge the hospitality of Stanford University in the fall 2013, where this work was initiated.



\end{document}

%% file: intro.tex
This article is about a random conductance problem on the integer lattice $\Z^d$. We regard $\Z^d$ as a graph with edge set $\B = \{ (x,z) \in \Z^d \times \Z^d \;|\;  |x-z|=1\}$. For edges $(x,z) \in \B$, we also write $x\sim z$.  We define the set of conductances $a(e)$ on the edges $e \in \B$  by $\Omega=[\lambda,1]^\B$ for some fixed $0<\lambda\le 1$. We equip $\Omega$ with the $\sigma$-algebra $\calF$ generated by cylinder sets and with a probability measure $\mathbb P$. We assume that $\Pm$ is invariant under the group of transformations $\tau_z:\Omega \to \Omega$ defined by $a(\cdot ) \mapsto a(z + \cdot)$ for all $z \in \Z^d$, and that this group of transformations acts ergodically on $(\Omega,\calF,\mathbb{P})$.

A realization $a \in \Omega$ is a countable set $\{a(e)\}_{e\in \B}$ of conductances and is called an environment. 
If $X_t$ is a continuous-time random walk in this random environment $a$ with infinitesimal generator
$$
\LL u(x)\,:=\,\sum_{z\sim x}a(x,z)(u(z)-u(x)),
$$
acting on functions $u:\Z^d \to \R$, then the rescaled walk $\e X_{t/\e^2}$ (parabolic scaling) converges to a Brownian motion in $\R^d$ with some covariance $2A_\ho$, as $\e \to 0$ (in a $\mathbb P$-annealed sense \cite{Kipnis-Varadhan-86}, and in a quenched sense \cite{Sidoravicius-Sznitman-04,Mathieu-08} if $\Pm$ is a product measure, for example). We refer to $A_\ho$ as the effective conductivity of the random environment, and it is characterized by
\begin{equation}\label{eq:hom-coeff}
\xi \cdot A_\ho \xi \,=\,\E{(\xi+\nabla \phi)\cdot \aa (\xi+\nabla \phi)(0)},
\end{equation}
for all $\xi \in \R^d$ with $|\xi|=1$, where $\phi$ and $\aa$ are defined as follows. For $x\in \Z^d$, we set $\aa(x):=\dig{a(x,x+\ee_1),\dots,a(x,x+\ee_d)}$ where $\ee_i$ is the $i^{th}$ standard basis vector in $\Rm^d$. We use $\nabla$ to denote the forward discrete gradient $(u:\Z^d\to \R)\mapsto (\nabla u:\Z^d \to \R^d)$ defined componentwise by $[\nabla u(x)]_i=u(x+\ee_i)-u(x)$ for all $i\in \{1,\dots,d\}$, 
and $\phi:\Z^d\times \Omega\to \R$ is defined, for almost every realization $a$, as the unique solution of
$$
\LL(\xi\cdot x+\phi(x,\aa))\,=\,0 \quad \text{in }\Z^d, \quad \phi(0,\aa)=0.
$$
This $\phi$ is called the corrector in the direction $\xi$. The generator $\LL$ can also be written as
\[
\LL u(x)=-\nabla^*\cdot \aa(x)\nabla u(x)
\]
for all $u:\Z^d\to \R$ and all $x\in \Z^d$, where $\nabla^*$ is the backward discrete gradient $(u:\Z^d\to \R)\mapsto (\nabla^* u:\Z^d \to \R^d)$ defined componentwise by $[\nabla^* u(x)]_i=u(x)-u(x-\ee_i)$ for all $i\in \{1,\dots,d\}$.
Since $\nabla \phi$ is a stationary function of $\aa$ (that is, $\nabla \phi(x+z,\aa)=\nabla \phi(x,\aa(\cdot+z))$ for all $x,z\in \Z^d$ and $\mathbb P$-almost every $\aa$), and the measure $\mathbb P$ is preserved by translation of $\aa$,  the expectation in \eqref{eq:hom-coeff} does not depend on the point where the argument is taken. Throughout this paper, $\Pm$ will be a product measure $\Pm = \pi^{\otimes \B}$ where $\pi$ is a Borel probability measure on $[\lambda, 1]$. In this case, the conductivity $A_\ho$ is actually a scalar multiple of the identity matrix. In particular, $A_\ho$ is independent of the unit vector $\xi$.  Henceforth, we will use $A_\ho$ to denote this scalar quantity.

\commentout{
Formula \eqref{eq:hom-coeff} is also at the core of the stochastic homogenization theory, which relates the almost sure large scale behavior of $\LL^{-1}$ to the large scale behavior of the inverse of the continuum Laplacian $(-A_\ho\triangle)^{-1}$. Indeed, \eqref{eq:hom-coeff} is used in combination with the ergodic theorem
}

The convergence of the rescaled walk $\e X_{t/\e^2}$ to a Brownian motion with covariance $2A_\ho$ may be regarded as a central limit theorem (quenched or annealed) for a random walk in a random environment. Alternatively, from the point of view of stochastic homogenization theory for PDEs \cite{Papanicolaou-Varadhan-79,Kozlov79}, $A_\ho$ is the effective or homogenized coefficient for the operator $\LL$ (or its continuum analogue). In this context, formula \eqref{eq:hom-coeff} for $A_\ho$ comes from an application of the ergodic theorem, in the form of the almost sure identity
\begin{equation}
A_\ho\,=\,\lim_{L\uparrow \infty} L^{-d}\sum_{x\in [0,L)^d\cap \Z^d}(\xi+\nabla \phi)\cdot \aa (\xi+\nabla \phi)(x). \label{Ahoergodic}
\end{equation}
Viewing \eqref{Ahoergodic} as an ergodic theorem, one may ask whether a central limit theorem also holds for an appropriate renormalization of the sum in \eqref{Ahoergodic}. This question is highly non trivial since the gradient of the corrector $\nabla \phi(\cdot,\aa)$ is a nonlinear nonlocal function of the random field $\aa$, and it is not clear how (and even whether) mixing properties of $\aa$ can be transmitted to $\nabla \phi$. Nevertheless, understanding the map $\aa \mapsto \nabla \phi(\cdot,\aa)$ and the statistics of quantities like the sum in (\ref{Ahoergodic}) is the key to the development of a quantitative theory of stochastic homogenization.

In addition to its theoretical interest, stochastic homogenization is used as a practical tool  in scientific computing to reduce the complexity of computations for highly heterogeneous materials, see for instance \cite{Efendiev-Hou-09,E-12} in the applied mathematics community and \cite{Torquato-02,KFGMJ-03} in the applied mechanics community.
In order to make practical use of the homogenization theory, one needs to approximate $A_\ho$. When the law of $\aa$ can be periodized on large scales (cf. discussion in \cite[Paragraph~3.2.3]{EGMN-12}), a general consensus is that $A_\ho$ should be approximated by periodization.
More precisely, following \cite[Section~5]{Gloria-Neukamm-Otto-14}, we define for $L \in \N$ the set $\T_L = [0,L)^d \cap \Z^d$ which we regard as a discrete torus of size $L$, and we define the edges $\B_L = \{ (x,x + \ee_i) \;|\; x \in \T_L, \; i = 1,\dots,d \}$. For all $x,z\in \T_L$ we say that $x\sim z$ if $|x-z|=1$, where $|\cdot|$ is the distance on the torus. We define a periodic extension of the coefficient $a$ by
\begin{equation}
a_L( x + L z, x + L z + \ee_i ) = a(x, x + \ee_i),  \quad \forall\; x \in \T_L, \;\; z \in \Z^d. \label{aLdef}
\end{equation}
Equivalently, 
\begin{equation}
\aa_L(x + L z) = \aa(x), \quad \forall\; x \in \T_L, \;\; z \in \Z^d. \label{aLdef2}
\end{equation}
The map $a \mapsto a_L$ defined in this way maps $\Omega$ to $\Omega$; under this map, the measure $\Pm$ pushes forward to a measure which concentrates on $L$-periodic coefficients. Therefore, we regard $a_L$ as an element of $\Omega_L = [\lambda,1]^{\B_L}$ which is equipped with the product measure $\Pm_L = \pi^{\otimes \B_L}$, $\pi$ being the measure on $[\lambda,1]$ associated with $\mathbb P=\pi^{\otimes \B}$. The periodic approximation $\phi_L$ of $\phi$ is now defined to be the unique $L$-periodic solution of the corrector equation
\begin{equation}\label{eq:corr-L}
-\nabla^*\cdot \aa_L (\xi+\nabla \phi_L(x))\,=\,0, \quad x \in \T_L,
\end{equation}
satisfying $\sum_{x \in \T_L} \phi_L(x) = 0$, and we define the averaged energy density $A_L$ (or effective conductance on the discrete torus $\T_L$) as
\begin{equation}\label{eq:hom-coeff-per}
A_L\, =\,L^{-d}\sum_{x \in \T_L}(\xi+\nabla \phi_L)\cdot \aa_L (\xi+\nabla \phi_L)(x).
\end{equation}
Observe that $A_L$ and $\phi_L$ depend implicitly on the unit vector $\xi$, whereas $A_\ho$ does not depend on $\xi$; nevertheless, the estimates of this manuscript will be uniform in $\xi$.
The contributions of Neukamm, Otto, and the first author \cite{Gloria-Neukamm-Otto-14} 
(which started with \cite{Gloria-Otto-09}, based upon \cite{Naddaf-Spencer-98}) established that
\begin{equation}\label{eq:scaling-GNO}
\expec{|A_L-A_\ho|^2} \,\lesssim\, L^{-d}.
\end{equation}
This estimate contains two statements: the control of the random error due to fluctuations of $A_L$ around $\expec{A_L}$ and the control of the systematic error $A_\ho-\expec{A_L}$ due to periodization of the law (see  \cite[Theorem~2]{Gloria-Neukamm-Otto-14}).
As empirical evidence suggests \cite{E-Yue-07},  not only in terms of scaling but also in terms of convergence in law the associated CLT is expected to hold true.
As a second step towards a CLT, the second author proved in the continuum setting \cite{Nolen-11, Nolen-14} (where the discrete elliptic equation on the torus is replaced by a divergence form linear elliptic PDE) that the fluctuations of $A_L-\expec{A_L}$ (rescaled by the square-root of the variance) are asymptotically normal.

%
%
%

The aim of this contribution is to go beyond the scaling \eqref{eq:scaling-GNO} and the asymptotic normality, and prove the associated full central limit theorem for $A_L$ by estimating the Kolmorogov distance of  $L^{\frac{d}{2}}(A_L-A_\ho)$ to a normal variable. Recall that if $X$ and $Y$ are two real-valued random variables, then the Kolmogorov distance between their distributions is
\[
d_K(X,Y) = \sup_{t \in \Rm} |\Pm(X \leq t) - \Pm(Y \leq t)|.
\]

\begin{theo}\label{th:main}
Let $d\ge 2$.
Let $\mathbb P$ be a nontrivial product measure, $\xi \in \R^d$ with $|\xi|=1$ be fixed, $A_L$ be given by \eqref{eq:hom-coeff-per}, and $A_\ho$ by \eqref{eq:hom-coeff}.
Then there exists $\sigma>0$ such that for all $L\in \N$ we have
\begin{equation}\label{eq:th-main}
d_{K} \Big(L^{\frac{d}{2}}\frac{A_L-A_\ho}{\sigma},\mathcal G \Big)\,\lesssim \, L^{-\frac{d}{2}}\log^dL,
\end{equation}
where $\mathcal G$ is a standard normal variable.

\qed
\end{theo}
\begin{rem}
The same statement also holds for the 1-Wasserstein distance instead of the Kolmogorov distance.
\end{rem}
On the one hand, this result gives a complete numerical analysis  of the widely-used periodization method in computational mechanics. On the other hand, it is the first quantitative central limit theorem in stochastic homogenization (a qualitative version has been proved in \cite{Biskup-Salvi-Wolff-14}, albeit in the perturbation regime of small ellipticity contrast).
It shows in particular that $L^{\frac{d}{2}}(A_L-A_\ho)$ fluctuates as the centered Gaussian of variance $\sigma^2$ up to a small error in Kolmogorov distance.
In view of the recent contributions \cite{Gloria-Otto-14,Gloria-Otto-14b,Armstrong-Smart-14,Gloria-Marahrens-14, Nolen-14}, we believe this result holds true for a continuum equation as well.

\medskip

The result of Theorem~\ref{th:main} is probabilistic in nature. However, the arguments of probability theory are wrapped up into some well-known estimates that hold for product measures, and most of the work relies on analysis.
Estimate~\eqref{eq:th-main} combines three statements:
\begin{itemize}
\item[(i)] a normal approximation estimate for $A_L$, 
\item[(ii)] an estimate of the systematic error $\E{A_L}-A_\ho$, 
\item[(iii)] an estimate of the rescaled variance $L^d\var{A_L}-\sigma^2$.
\end{itemize}
Let us quickly address (i) and (ii).
The normal approximation of $A_L$ was already unravelled by the second author in \cite{Nolen-11, Nolen-14}
in the continuum setting, however with the 1-Wasserstein distance instead of the Kolmogorov distance.
Its proof exploits the product structure of the law in two respects: the validity of the Efron-Stein inequality (in the spirit of the covariance estimate of Lemma~\ref{lem:covar}) and the refinement by Chatterjee \cite{Chatterjee-08, Chatterjee-09} of Stein's method.
To make this strategy work, moment bounds on $\nabla \phi_L$ are needed. These are obtained in \cite{Nolen-11} in the continuum setting following the approach of \cite{Gloria-Otto-09} (they are however suboptimal for $d=2$) and using averaged gradient bounds on the Green function. 
The adaptation of this approach to the discrete setting dealt with here is straightforward (and optimal for all $d\ge 2$) given the moment bounds of \cite{Gloria-Neukamm-Otto-14}, the annealed estimates on the Green function of \cite{Marahrens-Otto-13}, and the recent refinement of Chatterjee's method \cite{Chatterjee-08} by Lachi\`eze-Rey and Peccati \cite{LRP-15} for the Kolmogorov distance (see Theorem \ref{theo:normalapprox} below).
The result is the following proposition. We display its proof for completeness in Section~\ref{sec:normal}:
\begin{prop}\label{th:normal}
Let  $\sigma_L^2 := L^d \var{A_L}$. Then
\begin{equation}\label{eq:normal}
d_{K} \Big(L^{\frac{d}{2}}\frac{A_L-\E{A_L}}{ \sigma_L},\mathcal G \Big)\,\lesssim \,  L^{-\frac{d}{2}}\log L.
\end{equation}
\qed
\end{prop}
As already mentioned, the quantitative estimate of  $\E{A_L}-A_\ho$ is the object of \cite[Proposition~3]{Gloria-Neukamm-Otto-14}, which we recall here for completeness:
\begin{prop}\label{th:systematic}
\begin{equation}\label{eq:systematic}
|\E{A_L}-A_\ho|\,\lesssim \, L^{-d}\log^dL.
\end{equation}
\qed
\end{prop}
We now turn to the main achievement of this article: the definition of $\sigma$ and the estimate of $L^d \var{A_L}-\sigma^2$:
\begin{prop}\label{th:resc-var}
Let $\sigma_L^2:=L^d \var{A_L}$. There exists $\sigma>0$ such that for all $L\in \N$ we have
\begin{equation}\label{eq:resc-var}
|\sigma-\sigma_L| \,\lesssim \, L^{-\frac{d}{2}}\log^d L.
\end{equation}
\qed
\end{prop}
This result is proved in Section~\ref{sec:variance}. Its proof essentially builds upon ideas and results of \cite{Gloria-Otto-09,Gloria-Mourrat-10,Gloria-Neukamm-Otto-14,Marahrens-Otto-13,Wehr-Aiz-90,Nolen-14}.  Their combination is however subtle. The general strategy is as follows.
The only important probabilistic ingredient is a covariance estimate (in the spirit of the Efron-Stein inequality, see Lemma~\ref{lem:covar}). We shall also take advantage of a logarithmic-Sobolev inequality in the form obtained 
 in \cite{Marahrens-Otto-13} (see Lemma~\ref{lem:LSI}), but in view of our recent contributions \cite{Gloria-Otto-14,Gloria-Neukamm-Otto-14b} this is less essential (though it is convenient and holds true for product measures, as considered here).
As opposed to the error term $A_\ho-\expec{A_L}$ analyzed in \cite{Gloria-Neukamm-Otto-14} for which the limit $A_\ho$ is well-defined, the estimate of $\sigma^2-L^d\var{A_L}$ first requires us to define the asymptotic standard deviation $\sigma$. Although we are not now able to do this by a direct approach ($\sigma$ is formally given by a series which is not absolutely convergent), we shall argue that $L^d\var{A_L}$ is a Cauchy sequence. In order to compare $L^d\var{A_L}$ to ${L'}^d\var{A_{L'}}$ for two values $L'\ge L$, we shall localize the dependence of the correctors with respect to $\aa_L$ and $\aa_{L'}$ to some region of size not exceeding $L$ (therefore independent of $L'$) by adding a massive term to the corrector equation of magnitude $\mu$ (which we think to be of order $L^{-2}$), giving rise to an approximate corrector $\phi_{L,\mu}$, cf. \eqref{eq:corrector-eq-modif-per}. We are thus left with two contributions to estimate $L^d\var{A_L}-{L'}^d\var{A_{L'}}$: the error due to boundary conditions (whose influence is tamed by the massive term) and a systematic error due to the modification of the corrector equation by the massive term.
The latter is the most subtle error term. It involves the fourth moment $\expec{|\nabla \phi_{L,\mu}-\nabla \phi_L|^4}$. 
On the one hand, relying on the optimal bounds of the spectrum of $\mathcal L$ projected on the local drift
obtained in \cite{Gloria-Neukamm-Otto-14}, we shall estimate the second moment $\expec{|\nabla \phi_{L,\mu}-\nabla \phi_L|^2}$ by spectral theory. On the other hand, we shall upgrade the estimate of the second moment to the fourth moment by using the logarithmic-Sobolev inequality. To make this strategy work we rely on the sensitivity calculus introduced in \cite{Gloria-Otto-09} (which essentially measures how much the solution of a PDE depends on changes of the coefficients)
and on the optimal annealed estimates of the Green functions obtained in \cite{Marahrens-Otto-13}.
Since the CLT scaling naturally improves with dimension, our approximation $\phi_{L,\mu}$ of $\phi_L$ will not be precise enough in high dimensions. This will be solved by using higher order approximations obtained by Richardson extrapolation (with respect to $\mu$), as introduced in \cite{Gloria-Mourrat-10}.

\medskip

Theorem~\ref{th:main} is then a direct combination of Propositions~\ref{th:normal}---\ref{th:resc-var}.
Up to logarithmic corrections, the three terms \eqref{eq:normal}, \eqref{eq:systematic}, and \eqref{eq:resc-var} yield the same contribution to \eqref{eq:th-main}. If instead of $A_L$ we simply consider the arithmetic average of $a(e)$ over edges $e$ in $\B_L$ and we replace $A_\ho$ by $\E{a(e)}$ and $\sigma$ by $\var{a(e)}^{\frac{1}{2}}$, then  \eqref{eq:th-main} is standard and holds without the correction $\log^dL$. However, we believe that this logarithmic correction is optimal for $A_L$.

\medskip

Let us conclude this introduction by mentioning a couple of recent results towards Theorem~\ref{th:main},
besides the works \cite{Nolen-11, Nolen-14} by the second author in the continuum setting. In \cite{Biskup-Salvi-Wolff-14}, Biskup, Salvi, and Wolff proved that for $\lambda$ close enough to $1$ (that is, in the perturbative regime of small ellipticity contrast), there exists some asymptotic standard deviation $\sigma>0$ such that the following convergence in law holds:
$$
L^{\frac{d}{2}} (A_{L,dir}-\E{A_{L,dir}}) \to \mathcal G_\sigma,
$$
as $L \to \infty$, where $\mathcal G_\sigma$ is a normal variable with variance $\sigma^2>0$, and $A_{L,dir}$ is the averaged energy on $([0,L)\cap \Z)^d$ of the approximation of the corrector $\phi$ by using homogeneous Dirichlet boundary conditions (instead of the periodic boundary conditions used here). Their proof relies on Meyers' estimates for the bounds on the corrector and on Lindenberg-Feller type conditions for the CLT, so that their result addresses (i) and (iii) in a qualitative way and in a perturbative regime (whereas (ii) is expected to be of order $L^{-1}$ for Dirichlet boundary conditions, that is, larger in general than the fluctuations of order $L^{-\frac{d}{2}}$).
In \cite{Rossignol-12}, Rossignol proves that in the periodic setting the random variable
$$
L^{\frac{d}{2}}\frac{A_L - \expE[A_L]}{\sigma_L} 
$$
converges in law to a standard normal, as $L \to \infty$. This result is a qualitative version of Proposition~\ref{th:normal} (that is, without error estimate). We refer to \cite{Biskup-11} for a recent survey of several other problems related to random conductance models.

\medskip

%% file: variance.tex
\subsection{Structure of the proof and auxiliary results}

\noindent From a formal expansion, one expects the following formula to hold for $\sigma^2$:
\begin{equation}\label{eq:formal-sigma}
\sigma^2\,=\,\sum_{z\in \Z^d} \cov{(\xi+\nabla \phi)\cdot \aa(\xi+\nabla \phi)(z)}{(\xi+\nabla \phi)\cdot \aa(\xi+\nabla \phi)(0)},
\end{equation}
where $\nabla \phi(0)\,:=\,\lim_{L\uparrow \infty} \nabla \phi_L(0)$ in $L^2(\Omega)$. Yet, we are not able to prove that this formula makes sense. Indeed, this sum displays cancellations which cannot be unravelled by using the triangle inequality (each term is expected to scale as the mixed second gradient of the elliptic Green function, the sum of which converges but is not absolutely convergent). Instead, we will show that $\sigma_L$ is a Cauchy sequence which has a positive limit $\sigma$ as $L \to \infty$.

To circumvent the difficulty in making sense of \eqref{eq:formal-sigma}, we make use of a regularization of the corrector equation by a massive term of magnitude $\mu>0$, and let $\phi_\mu:\Z^d\times \Omega \to \R$ be the unique bounded solution of 
\begin{equation}\label{eq:corrector-eq-modif}
\mu\phi_\mu(x,\aa)-\nabla^*\cdot \aa(x)(\xi+\phi_\mu(x,\aa))\,=\,0 \quad x \in \Z^d, 
\end{equation}
(such a solution is measurable on $\Omega$ and can be defined by the Green representation formula for all $\aa \in \Omega$ --- and not only $\mathbb P$ almost surely).
Then the formula 
\begin{equation}\label{eq:formal-sigma-mu}
\sigma_\mu^2\,:=\;\sum_{z\in \Z^d} \cov{(\xi+\nabla \phi_\mu)\cdot \aa(\xi+\nabla \phi_\mu)(z)}{(\xi+\nabla \phi_\mu)\cdot \aa(\xi+\nabla \phi_\mu)(0)}
\end{equation}
makes sense (we shall prove that the covariances are summable). 
Likewise, for all $L\in \N$ we consider the unique $[0,L)^d$-periodic solution $\phi_{\mu,L}:\Omega \times \Z^d \to \R$ of
\begin{equation}\label{eq:corrector-eq-modif-per}
\mu\phi_{\mu,L}(x,\aa)-\nabla^*\cdot \aa_L(x,\aa) (\xi+\nabla\phi_{\mu,L}(x,\aa))\,=\,0\quad \text{in }\Z^d, 
\end{equation}
where $\aa_L(\cdot,\aa)$ is the periodic extension of $\aa$ defined at (\ref{aLdef2}).

In a second step we compare $\nabla \phi_{\mu,L}$ to $\nabla \phi_L$ and $\nabla\phi_{\mu,L}$ to $\nabla\phi_\mu$, and then optimize $\mu$ with respect to $L$ to show that $\sigma_L$ is Cauchy as $L \to \infty$. The fact that the limit is positive is a consequence of the following proposition.
\begin{prop}\label{prop:lower-bound}
If $\mathbb P$ is a nontrivial product measure, then there exists $\underline{\sigma}>0$ such that 
for all $L\in \N$,
$$
\var{A_L}\,\geq \, \underline{\sigma}L^{-d},
$$
so that $\sigma_L=\Big(L^d\var{A_L}\Big)^{\frac{1}{2}}\ge \underline\sigma$.
\qed
\end{prop}
Although we do not prove that the limit $\sigma$ of $\sigma_L$ coincides with \eqref{eq:formal-sigma}, this optimization in $\mu$ is a way to take implicitly into account the cancellations in \eqref{eq:formal-sigma}.

From a technical point of view, it is not enough to replace $\nabla \phi_L$ by $\nabla \phi_{L,\mu}$ in large dimensions since the scaling of the CLT improves wrt the dimension whereas $\expec{|\nabla  \phi_L-\nabla \phi_{L,\mu}|^2}$ has a limited precision (which turns out to be $\mu^{d}$ up to dimension 4, where it saturates at $\mu^4$). To enhance the convergence rate and make the approximation error smaller than the CLT scaling, we shall use Richardson extrapolation.
As opposed to \cite{Gloria-Neukamm-Otto-14}, we use the Richardson extrapolation at the level of the corrector rather than at the level of the homogenized coefficients.
\begin{defi}[Richardson extrapolations]\label{def:Richardson}
For all $\mu>0$ and $L\in \N$, the sequences $\{\phi_{k,\mu}\}_{k\in \N}, \{\phi_{L,k,\mu}\}_{k\in \N}:\Omega \times \Z^d\to \R$ are defined by $\phi_{1,\mu}:=\phi_\mu$, $\phi_{L,1,\mu}:=\phi_{L,\mu}$,
 and by the general induction formula
\begin{equation}\label{eq:Richardson1}
\phi_{(L,)k+1,\mu}\,:=\,\frac{1}{2^k-1}(2^k \phi_{(L,)k,\frac{\mu}{2}}-\phi_{(L,)k,\mu}).
\end{equation}
We then set
\begin{eqnarray*}
A_{k,\mu}&:=& \expec{(\xi+\nabla \phi_{k,\mu})\cdot \aa (\xi+\nabla \phi_{k,\mu})(0)},\\
A_{L,k,\mu}&:=&
L^{-d}\sum_{x \in \T_L} (\xi+\nabla \phi_{L,k,\mu})\cdot \aa_L (\xi+\nabla \phi_{L,k,\mu})(x),
\end{eqnarray*}
\qed
\end{defi}
Proposition~\ref{th:resc-var} is then a consequence of 
 Proposition \ref{prop:lower-bound} and of the following two propositions.
\begin{prop}\label{prop:LtoLmu}
For all $\mu>0$ and all $k,L\in \N$, set
\begin{equation}\label{eq:def-sig-Lkmu}
\sigma_{L,k,\mu}^2\,:=\,L^d \var{A_{L,k,\mu}}.
\end{equation}
Then for all $k > \frac{d}{4}$ we have
\begin{equation}\label{eq:LtoLmu}
|\sigma^2_L-\sigma^2_{L,k,\mu}|\,\lesssim \, (\mu L)^{\frac{d}{2}}+ (\mu L)^d.
\end{equation}
\qed
\end{prop}
\begin{prop}\label{prop:var-resc-mu}
For all $\mu>0$ and $k\in \N$, there exists $\sigma_{k,\mu} \in \R^+$, such that for all $L\in \N$,
\begin{equation}\label{eq:resc-var-mu}
|\sigma^2_{L,k,\mu}-\sigma^2_{k,\mu}| \,\lesssim \, L^d \log(2+\sqrt{\mu}L) e^{-c\sqrt{\mu}L},
\end{equation}
where $c>0$ only depends on $\lambda$ and $d$, and the multiplicative constant depends on $k$, next to $\lambda$ and $d$.
\qed
\end{prop}
\noindent Indeed, for all $L\in \N$ and all $L< L' \leq 2L$, we have for all $\mu>0$ and $k > \frac{d}{4}$
\begin{eqnarray*}
|\sigma^2_L-\sigma^2_{L'}|&\leq&|\sigma^2_L-\sigma^2_{L,k,\mu}|+|\sigma^2_{L,k,\mu}-\sigma^2_{k,\mu}|+|\sigma^2_{k,\mu}-\sigma^2_{L',k,\mu}|+|\sigma^2_{L',k,\mu}-\sigma^2_{L'}|
\\
&\stackrel{\eqref{eq:LtoLmu}\&\eqref{eq:resc-var-mu}}{\lesssim}& (\mu L)^{\frac{d}{2}}+ (\mu L)^d+L^d \log(2+\sqrt{\mu}L) e^{-c\sqrt{\mu}L}.
\end{eqnarray*}
Optimizing this inequality with respect to $\mu$ yields (by taking $\mu=K\Big(\frac{1}{L}\log \big(\frac{L}{\log L}\big)\Big)^2$ for $K$ large enough):
\begin{equation}\label{eq:diff-LL'}
|\sigma^2_L-\sigma^2_{L'}|\,\lesssim\, L^{-\frac{d}{2}}\log^dL.
\end{equation}
Let now $M,L\in \N$ with $M\geq L$, and let $m\in \N_0$ be such that $2^{m}L< M\leq 2^{m+1}L$.
Then, by the triangle inequality, we have
\begin{eqnarray}
|\sigma^2_L-\sigma^2_{M}|&\leq&
\sum_{l=0}^{m-1}|\sigma^2_{2^lL}-\sigma^2_{2^{l+1}L}|
+|\sigma^2_{2^{m}L}-\sigma^2_M|\nonumber
\\
&\stackrel{\eqref{eq:diff-LL'}}{\lesssim}& \sum_{l=0}^{\infty} 
(2^lL)^{-\frac{d}{2}}\log^d(2^lL) \,\lesssim \, L^{-\frac{d}{2}}\log^dL,\label{eq:Cauchy-quant}
\end{eqnarray}
so that $\sigma^2_L$ is a Cauchy sequence, from which we deduce the existence of the limit $\sigma^2$ and the estimate $|\sigma^2-\sigma_L^2|\lesssim L^{-\frac{d}{2}}\log^d L$.
Proposition~\ref{prop:lower-bound} then implies that $\sigma^2>0$, so that  \eqref{eq:Cauchy-quant} turns into the desired estimate~\eqref{eq:resc-var}:
$$
L^{-\frac{d}{2}}\log^dL\,\gtrsim \, |\sigma^2-\sigma_L^2|\,=\,|\sigma-\sigma_L|(\sigma+\sigma_L) \,\geq \, 2\underline{\sigma}|\sigma-\sigma_L|.
$$

\subsection{Proof of Proposition~\ref{prop:lower-bound}}

A version of this was proved by Wehr \cite{Wehr-97} under an additional regularity assumption about the law of $a(e)$ and with Dirichlet boundary conditions for $\phi_{L}$, but without the uniform lower bound assumption $a(e) \geq \lambda > 0$. See also \cite{Nolen-11} for a continuum version of Proposition \ref{prop:lower-bound}. The starting point is the lower bound
\begin{equation}
\var	{X} \geq \sum_{j \in \mathbb{B}_L} \var{ \;\expE[X \;|\; a(j)]\; } \label{varlower2}
\end{equation}
that holds for all $X \in L^2(\Omega_L,\Pm_L)$, where $\expE[X \;|\; a(j)]$ is the conditional expectation of $X$, given conductivity $a(j)$ on the edge $j \in \B_L$, that we shall
apply to the averaged energy $A_L$ of the corrector $\phi_L$. 
This lower bound follows from the fact that the random variables $\{ a(j) \}_{j \in \mathbb{B}_L}$ are independent under the product measure $\Pm_L = \pi^{\otimes B_L}$ on $\Omega_L$ (cf. \cite[Proposition 3.1]{Wehr-Aiz-90}).

Due to stationarity and the fact that 
$$
\expE[A_L] = L^{-d} \sum_{x \in \T_L} \expE[ (\nabla \phi_{L} + \xi) \cdot \aa_L (\nabla \phi_{L} + \xi)(x)] \in [\lambda,1],
$$
it follows that for some $i \in \{1,\dots,d\}$,
$$
\expE[ (\nabla \phi_{L}(j) + \xi \cdot \ee_i)^2] \geq \lambda/d,
$$
where $j$ is the edge $j = (0,0 + \ee_i)$. We claim that 
\begin{equation}
\var{ \;\expE[A_L \;|\; a(j) ] \;} \gtrsim L^{-2d} \label{varlower1}
\end{equation}
as $L \to \infty$. This fact, the lower bound \eqref{varlower2}, and the stationarity of $a$ imply that for some $\underline{\sigma} > 0$, $\var{A_L} \geq \underline{\sigma} L^{-d}$ holds for all $L \in \N$. 

Now we establish \eqref{varlower1}. With $g(\alpha) = \expE[A_L \;|\; a(j) = \alpha]$, we write the variance as
$$
\var{ \;\expE[A_L \;|\; a(j)] \;} = \frac{1}{2} \int_{[\lambda, 1]^2} (g(\alpha) - g(\alpha'))^2 d \pi(\alpha) \otimes d\pi(\alpha')
$$
where $\pi$ is the law of $a(j)$ on $[\lambda,1]$, and $\pi \otimes \pi$ is the product measure on $[\lambda,1]^2$. Hence, provided $g$ is differentiable,
\br
\var{ \expE[A_L \;|\; a(j)] } & \geq & \frac{1}{2} \left( \inf_{\alpha \in [\lambda,1]} g'(\alpha) \right) \int_{[\lambda, 1]^2} (\alpha - \alpha')^2  d \pi(\alpha) \otimes d\pi(\alpha') \no \\
& = & \frac{1}{2} \left( \inf_{\alpha \in [\lambda,1]} g'(\alpha) \right)^2 \var{a(j)}. \label{vargprime}
\er
Let $\hat a$ and $a$ coincide except at edge $j$, and denote by $\aa_L^j$ and $\aa_L$ the associated coefficient fields, $\phi_L^j$ and $\phi_L$ the associated correctors, and $A_L^j$ and $A_L$ the associated averaged energies.
By symmetry of $\aa_L$ and $\aa_L^j$ and the corrector equation, we then have
\begin{eqnarray}
L^d (A_L^j-A_L)&=&\sum_{x \in \T_L} (\nabla \phi_{L}^j +  \xi) \cdot \aa_L^j (\nabla \phi_{L}^j +  \xi)(x)-\sum_{x \in \T_L} (\nabla \phi_{L}+  \xi) \cdot \aa_L (\nabla \phi_{L} +  \xi)(x) \nonumber
\\
&=& \sum_{x \in \T_L} (\nabla \phi_{L}^j +  \xi) \cdot \aa_L^j (\nabla \phi_{L} +  \xi)(x)-\sum_{x \in \T_L} (\nabla \phi_{L}^j+  \xi) \cdot \aa_L (\nabla \phi_{L} +  \xi)(x) \nonumber\\
&=&\sum_{x \in \T_L} (\nabla \phi_{L}^j +  \xi) \cdot (\aa_L^j-\aa_L) (\nabla \phi_{L} +  \xi)(x).
\label{eq:antoine-shorter}
\end{eqnarray}
Therefore, with $j = (0, 0 + \ee_i)$, $g(\alpha) = \expE[A_L \;|\; a(j) = \alpha]$ is differentiable and its derivative is given by
$$
g'(\alpha) = L^{-d} \expE[ (\nabla \phi_{L}(j) + \xi \cdot \ee_i)^2 \;|\; a(j) = \alpha].
$$
We claim that there is a constant $c > 0$, independent of $L$ and $j = (0,j + \ee_i)$, such that
\begin{equation}
g'(\alpha) \geq c g'(\beta), \quad \text{for all} \;\; \alpha,\beta \in [\lambda,1]. \label{gplower}
\end{equation}
Therefore,
$$
g'(\alpha) \geq c \int_{[\lambda,1]} g'(\beta) \,d \pi(\beta) = c L^{-d} \expE[ (\nabla \phi_{L}(j) + \xi \cdot \ee_i)^2 ] \geq L^{-d}\frac{c \lambda}{d}.
$$
Returning to \eqref{vargprime}, we conclude that
$$
\var{ \;\expE[A_L \;|\; a(j)] \;} \gtrsim L^{-2d} \var{a(j)}.
$$
Thus, as long as $\pi$, the law of $a(j)$, is nontrivial, the desired bound \eqref{varlower1} holds.

It remains to prove \eqref{gplower}, which is a discrete version of \cite[Lemma 2.2]{Nolen-11}. Suppose $\aa_L$ and $\aa_L'$ are two periodic coefficients which agree everywhere except on edges $j \in M \subset \mathbb{B}_L$. Let $\phi_{L} = \phi_{L}(x,\aa_L)$ and $\phi_{L}' = \phi_{L}(x,\aa_L')$ be the associated $L$-periodic correctors. Then the function $v = \phi_{L} - \phi_{L}'$ satisfies
$$
- \nabla^* \cdot \aa_L'  \nabla v  = - \nabla^* \cdot  (\aa_L' - \aa_L) (\nabla \phi_{L} + \xi), \quad \text{in}\;\T_L.
$$
By the energy estimate, we conclude that
$$
\sum_{j \in \mathbb{B}_L} (\nabla v(j))^2 \lesssim \sum_{j \in M}(\nabla \phi_{L}(j) + \xi \cdot \ee_i)^2.
$$
From this and the triangle inequality it follows that
$$
\sum_{j \in M} (\nabla \phi_{L}'(j) + \xi \cdot \ee_i)^2 \lesssim \sum_{j \in M} (\nabla \phi_{L}(j) + \xi \cdot \ee_i)^2
$$
which implies \eqref{gplower}.

\subsection{Proof of Proposition~\ref{prop:LtoLmu}}

By definition of $\sigma^2_L$ and $\sigma^2_{L,k,\mu}$, 
\begin{eqnarray}
\lefteqn{L^{-d}|\sigma^2_L-\sigma^2_{L,k,\mu}|} \nonumber\\
&=&|\var{A_L}-\var{A_{L,k,\mu}}| \nonumber\\
&=&|\E{(A_L-A_{L,k,\mu}-\E{A_L}+\E{A_{L,k,\mu}})(A_L+A_{L,k,\mu}-\E{A_L}-\E{A_{L,k,\mu}})}| \nonumber\\
&\lesssim & \E{(A_L-A_{L,k,\mu})^2}^{\frac{1}{2}}\Big( \var{A_L}^{\frac{1}{2}}+\E{(A_L-A_{L,k,\mu})^2}^{\frac{1}{2}}\Big).
\label{eq:diff-sigma_L}
\end{eqnarray}
We reformulate the first term of the RHS.
By the weak form of the corrector equation \eqref{eq:corr-L} for $\phi_L$ and symmetry of $\aa_L(\aa)$,
\begin{eqnarray}
\lefteqn{L^d(A_L-A_{L,k,\mu})} \no\\
&=&\sum_{\T_L}(\xi+\nabla \phi_L)\cdot\aa_L(\xi+\nabla \phi_L)-\sum_{\T_L}(\xi+\nabla \phi_{L,k,\mu})\cdot\aa_L(\xi+\nabla \phi_{L,k,\mu}) \no \\
&=&\underbrace{\sum_{\T_L}\nabla (\phi_L-\phi_{L,k,\mu})\cdot \aa_L(\xi+\nabla \phi_L)}_{\dps =\,0}+\sum_{\T_L}(\xi+\nabla \phi_{L,k,\mu})\cdot\aa_L(\nabla \phi_L-\nabla \phi_{L,k,\mu}) \no \\
&
{=}&\underbrace{\sum_{\T_L}\nabla (\phi_{L,k,\mu}-\phi_L)\cdot\aa_L(\xi+\nabla \phi_L)}_{\dps =\,0}+
\sum_{\T_L}(\nabla \phi_L-\nabla \phi_{L,k,\mu})\cdot \aa_L(\xi+\nabla \phi_{L,k,\mu}) \no \\
&=&-\sum_{\T_L}\nabla (\phi_L-\phi_{L,k,\mu})\cdot\aa_L\nabla (\phi_L-\phi_{L,k,\mu}). \label{eq:diff-sigma_L2}
\end{eqnarray}
Expanding the square of the RHS of \eqref{eq:diff-sigma_L2}, using Cauchy-Schwarz' inequality, and stationarity, this yields
\begin{equation*}
\E{(A_L-A_{L,k,\mu})^2}\,\leq\, \E{|\nabla  \phi_L-\nabla  \phi_{L,k,\mu}|^4},
\end{equation*}
so that \eqref{eq:diff-sigma_L} turns into
\begin{equation}
|\sigma^2_L-\sigma^2_{L,k,\mu}| \,
\lesssim\,L^d  \E{|\nabla  \phi_L-\nabla  \phi_{L,k,\mu}|^4}^{\frac{1}{2}}
\Big(\var{A_L}^{\frac{1}{2}}+\E{|\nabla  \phi_L-\nabla  \phi_{L,k,\mu}|^4}^{\frac{1}{2}}\Big).
\label{eq:diff-sigma_L2-2}
\end{equation}
By \cite[Proposition~2]{Gloria-Neukamm-Otto-14}, $\var{A_L}\,\lesssim \,L^{-d}$, and it remains to bound the first term of the RHS of \eqref{eq:diff-sigma_L2-2}, that is, the fourth moment of $|\nabla  \phi_L-\nabla  \phi_{L,k,\mu}|$.
The desired bound \eqref{eq:LtoLmu} in Proposition~\ref{prop:LtoLmu} will follow from \eqref{eq:diff-sigma_L2-2} and the estimate 
\begin{equation}\label{eq:revis-1}
\E{|\nabla  \phi_L-\nabla  \phi_{L,k,\mu}|^4}^{\frac{1}{2}}\,\lesssim\, \mu^{\frac{d}{2}}.
\end{equation}
To prove the latter it will be convenient to write this difference in the 
form of
$$
\nabla  \phi_{L,k,\mu}-\nabla  \phi_L\,=\,\int_{0}^\mu \nabla \partial_{\hat\mu}   \phi_{L,k,\hat\mu} d\hat\mu,
$$
where the identity holds as functions of $\T_L \to \R^d$. Indeed, on the one hand, $\mu \mapsto  \phi_{L,k,\mu}$ is analytic 
from $(0,\infty)$ to $L^\infty(\T_L ,\R)$, and on the other hand $\phi_L=\lim_{\mu \downarrow 0}\phi_{L,k,\mu}$ for all $k\in \N$.
The estimate \eqref{eq:revis-1} is indeed a consequence of 
\begin{equation}\label{eq:moment-4.0}
\E{|\nabla \partial_{\hat\mu}   \phi_{L,k,\hat\mu}|^4}^{\frac{1}{4}}\,\lesssim \,\hat \mu^{\frac{d}{4}-1},
\end{equation}
by integration between $0$ and $\mu$ and the triangle inequality:
$$
\E{|\nabla  \phi_L-\nabla  \phi_{L,k,\mu}|^4}^{\frac{1}{4}}
\,\stackrel{\triangle-\text{ineq.}}{\leq}\, \int_0^\mu \E{|\partial_{\hat\mu}  \nabla \phi_{L,k,\hat\mu}|^4}^{\frac{1}{4}}d\hat \mu \,\stackrel{\eqref{eq:moment-4.0}}{\lesssim}\, \mu^{\frac{d}{4}}.
$$
The strategy to prove \eqref{eq:moment-4.0} is to start with the second moment (using spectral theory) and increase the integrability by using a logarithmic-Sobolev inequality.

We now address the estimate of the second moment using spectral theory.
\begin{lemma}\label{lem:moment-2}
For all $L\in \N$, $d\ge 2$, $\mu>0$ and $k\in \N$,
\begin{equation}\label{eq:phi_L-phi_Lkmu-2}
\E{|\nabla \partial_{\mu}   \phi_{L,k,\mu}|^2}^\frac{1}{2}\,\lesssim\, 
\left\{
\begin{array}{lll}
k<\frac{d}{4}&:&\mu^{k-1},\\
k=\frac{d}{4}&:&\mu^{\frac{d}{4}-1}|\log \mu|^{\frac{1}{2}},\\
k>\frac{d}{4}&:&\mu^{\frac{d}{4}-1},
\end{array}
\right\}
\end{equation}
where the multiplicative constant is independent of $\mu$ and $L$.
\qed
\end{lemma}
Before we give the proof of Lemma~\ref{lem:moment-2} based on spectral theory, let us give the intuition for this scaling when $k=1$. 
First of all, since the bounds are polynomial, we may expect that 
$\E{|\nabla \partial_{\mu}   \phi_{L,k,\mu}|^2}^\frac{1}{2}\,\lesssim \, \mu^{-1}\E{|\nabla  \phi_{L,k,\mu}|^2}^\frac{1}{2}$.
The difference $\phi_L-\phi_{L,\mu}$ satisfies the following equation on $\T_L$
$$
\mu (\phi_L-\phi_{L,\mu})-\nabla^*\cdot \aa_L \nabla (\phi_L-\phi_{L,\mu}) \,=\,\mu \phi_L.
$$
This yields the a priori estimate
$$
\expec{|\nabla (\phi_L-\phi_{L,\mu})|^2} \,\lesssim\,\mu \expec{\phi_L(\phi_L-\phi_{L,\mu})},
$$
the RHS of which we can write as a covariance $\cov{\phi_L}{\phi_L-\phi_{L,\mu}}$. This covariance is the source of cancellations, which can be unravelled either by the use of spectral theory or by the use of the covariance estimate of Lemma~\ref{lem:covar}  below (this more intuitive approach is carried out in \cite{Gloria-Otto-09b} to estimate $\expec{|\nabla \phi_L-\nabla \phi_{L,\mu}|^2}$).
In particular, except in dimension $d=2$, the Cauchy-Schwarz inequality 
$\expec{\phi_L(\phi_L-\phi_{L,\mu})}\leq \var{\phi_L}^\frac{1}{2}\var{\phi_L-\phi_{L,\mu}}^{\frac{1}{2}}$ would not yield the right scaling. As opposed to the direct approach based on the covariance estimate, the spectral approach allows one to treat all $k$ at once.
\begin{proof}[Proof of Lemma~\ref{lem:moment-2}]
Estimate~\eqref{eq:phi_L-phi_Lkmu-2} is a consequence of the optimal bound on the spectral exponents
conjectured in \cite{Mourrat-10} and proved in 
\cite{Gloria-Neukamm-Otto-14} on the elliptic operator $-\nabla^*\cdot \aa_L \nabla$ in probability.
We recall this spectral result in the first step, and then prove the claim by induction in the last two steps.

\medskip

\step{1} Spectral theory.

\noindent  
We follow the approach introduced in \cite{Papanicolaou-Varadhan-79} in the continuum setting.
Let $L\in \N$. Since the measure $\mathbb{P}_L = \pi^{\otimes \B_L}$ on $\Omega_L = [\lambda,1]^{\B_L}$ is invariant by integer shifts of the torus, we can define a difference calculus on the Hilbert space $L^2(\Omega_L,\mathbb P_L)$.
We define the forward discrete derivative $D$ as the map $(\psi:\Omega_L \to \R)\mapsto (D \psi:\Omega_L \to \R^d)$ defined componentwise by $[D \psi(\aa_L)]_i=\psi(\aa_L(\cdot+\ee_i))-\psi(\aa_L)$ for all $i\in \{1,\dots,d\}$, 
and the backward discrete derivative $D^*$ as the map $(\psi:\Omega_L \to \R)\mapsto (D^* \psi:\Omega_L \to \R^d)$ defined componentwise by $[D^* \psi(\aa_L)]_i=\psi(\aa_L)-\psi(\aa_L(\cdot-\ee_i))$ for all $i\in \{1,\dots,d\}$.
For all $\mu \ge 0$ we then consider the unique weak solution $\tilde \phi_{L,\mu} \in  L^2(\Omega_L,\mathbb P_L)$ of
$$
\mu \tilde \phi_{L,\mu}(\aa_L)-D^*\cdot \aa_L (\xi+D\tilde \phi_{L,\mu}(\aa_L))\,=\,0,
$$
which exists by the Riesz representation theorem (for $\mu=0$, we denote $\tilde \phi_{L,0}$ by 
$\tilde \phi_L$). Richardson extrapolations $\tilde \phi_{L,k,\mu}$ are defined accordingly.

Note that the stationary extension $\Z^d\times \Omega_L \ni (x,\aa_L) \mapsto \tilde \phi_{L,\mu}(\aa_L(\cdot+x))$ coincides $\mathbb P_L$-almost surely with the solution $\phi_{L,\mu}(x,\aa_L)$ of \eqref{eq:corrector-eq-modif-per} (where $\aa_L$ also denotes the periodic extension on $\Z^d$ of $\aa_L \in \Omega_L$), and likewise for derivatives with respect to $\mu$. Since $\mathbb P=\pi^{\otimes \B}$ and $\mathbb P_L=\pi^{\otimes \B_L}$, if $\chi \in L^2(\Omega_L,\mathbb P_L)$ then $\chi\in L^2(\Omega,\mathbb P)$, and
$$
\E{\chi^2}\,=\,\EL{\chi^2},
$$
which yields the starting point of this proof:
\begin{equation}\label{eq:refo-phi_L-phi_Lkmu-2}
\E{|\nabla \partial_{\mu}   \phi_{L,k,\mu}|^2}\,=\,\EL{|D \partial_{\mu} \tilde \phi_{L,k,\mu}|^2}.
\end{equation}

Since $\mathbb L_L=-D^*\cdot \aa_L D$ is a bounded non-negative self-adjoint linear operator on $L^2(\Omega_L,\mathbb P_L)$, it admits a spectral resolution:
$$
\mathbb L_L\,=\,\int_0^\infty \nu P_L(d\nu).
$$
Set $\mathfrak{e}:=D^*\cdot \aa_L \xi$. As proved in \cite[Corollary~1]{Gloria-Neukamm-Otto-14}, we have for all 
$\hat \nu\ge 0$,
\begin{equation}\label{eq:spectral-exp}
\int_0^{\hat \nu} \EL{\mathfrak{e} P_L(d\nu)\mathfrak{e}}\,\lesssim\, {\hat \nu}^{\frac{d}{2}+1},
\end{equation}
where the multiplicative constant does not depend on $\hat \nu$, $L$, or $\xi$ ($|\xi| = 1$).
This is the key to the proof of \eqref{eq:phi_L-phi_Lkmu-2}.

\medskip

\step{2} Spectral formula for the RHS of \eqref{eq:refo-phi_L-phi_Lkmu-2}: we claim that
\begin{equation}\label{eq:sp-fo-phi_L-phi_Lkmu-2}
\EL{|D\partial_{\mu} \tilde \phi_{L,k,\mu}|^2}\,\lesssim \, \int_{\R^+} \frac{\mu^{2(k-1)}}{(\mu+\nu)^{2k+1}}\EL{\mathfrak{e} P_L(d\nu)\mathfrak{e}},
\end{equation}
where the multiplicative constant does depend on $k$, but not on $\mu$ and $L$.

By definition, for all $\mu\ge 0$ and $k\in \N$, 
$$
\tilde \phi_{L,k,\mu}\,=\, \psi_{k,\mu}(\mathbb{L}_L)\mathfrak{e},
$$
where 
$$ \psi_{1,\mu}:\R^+\to \R^+, \nu\mapsto  \psi_{1,\mu}(\nu)=\frac{1}{\mu+\nu},$$
and for all $k\in \N$, 
$$ \psi_{k+1,\mu}:\R^+\to \R^+, \nu\mapsto  \psi_{k+1,\mu}(\nu)=\frac{1}{2^k-1}(2^k  \psi_{k,\frac{\mu}{2}}(\nu)- \psi_{k,\mu}(\nu)).$$
Likewise,
$$
\partial_{\mu} \tilde \phi_{L,k,\mu}\,=\,\partial_\mu  \psi_{k,\mu}(\mathbb{L}_L)\mathfrak{e}.
$$
Hence, by ellipticity of $\aa_L$, and the spectral theorem, we have
\begin{eqnarray*}
\EL{|D\partial_\mu \tilde \phi_{L,k,\mu}|^2}&\lesssim &\EL{D\partial_\mu \tilde \phi_{L,k,\mu}\cdot \aa_LD\partial_\mu \tilde \phi_{L,k,\mu}}
\\
&=&
\EL{\mathfrak{e}\partial_\mu  \psi_{k,\mu}(\mathbb{L}_L)\mathbb{L}_L\partial_\mu  \psi_{k,\mu}(\mathbb{L}_L)\mathfrak{e}}
\\
&=&\int_0^\infty \nu\big(\partial_\mu  \psi_{k,\mu}(\nu)\big)^2\EL{\mathfrak{e} P_L(d\nu)\mathfrak{e}}.
\end{eqnarray*}
The claim \eqref{eq:sp-fo-phi_L-phi_Lkmu-2} then follows from the following identity,
that we shall prove by induction,
\begin{equation}\label{eq:ind-spectr-integ}
\partial_\mu  \psi_{k,\mu}(\nu)\,=\, \mu^{k-1} \frac{p_k(\mu,\nu)}{ \prod_{i=0}^{k-1} (2^{-i}\mu+\nu)^2},
\end{equation}
where $p_k(\mu,\nu)=\sum_{j=0}^{k-1} a_j \mu^j\nu^{k-1-j}$ is the sum of monomials 
of total degree $k-1$.
For $k=1$,
$$
\partial_\mu  \psi_{1,\mu}(\nu)\,=\,-\frac{1}{(\mu+\nu)^2},
$$
and \eqref{eq:ind-spectr-integ} holds with $p_1\equiv-1$.
Assume that \eqref{eq:ind-spectr-integ} holds at step $k\in \N$.
Note that $\frac{\partial}{\partial \mu}\Big(  \psi_{k,\frac{\mu}{2}}(\nu)\Big)=\frac{1}{2}\partial_\mu  \psi_{k,\frac{\mu}{2}}(\nu)$.
We then have 
\begin{eqnarray*}
\partial_\mu  \psi_{k+1,\mu}(\nu)&=&
\frac{1}{2^k-1}\Big(2^k \frac{\partial}{\partial \mu}\Big( \psi_{k,\frac{\mu}{2}}(\nu)\Big)-\partial_\mu  \psi_{k,\mu}(\nu)\Big)
\\
&=&\frac{1}{2^k-1}\Big(2^{k-1} (\frac{\mu}{2})^{k-1} \frac{p_k(\frac{\mu}{2},\nu)}{ \prod_{i=0}^{k-1} (2^{-i-1}\mu+\nu)^2}
-\mu^{k-1} \frac{p_k(\mu,\nu)}{ \prod_{i=0}^{k-1} (2^{-i}\mu+\nu)^2} \Big)
\\
&=&\frac{\mu^{k-1}}{2^k-1}
\frac{p_k(\frac{\mu}{2},\nu)(\mu+\nu)^2-p_k(\mu,\nu)(2^{-k}\mu+\nu)^2}{ \prod_{i=0}^{k} (2^{-i}\mu+\nu)^2}.
\end{eqnarray*}
By the induction assumption,  $p_k(\frac{\mu}{2},\nu)(\mu+\nu)^2-p_k(\mu,\nu)(2^{-k}\mu+\nu)^2$
is the sum of monomials $\mu^j\nu^i$ of total degree $i+j=k+2$. In addition the coefficient of the term $\nu^{k+2}$ vanishes, so that this polynomial is divisible by $\mu$. We may then set
$$
p_{k+1}(\mu,\nu)\,:=\,\frac{1}{\mu(2^k-1)}\big(p_k(\frac{\mu}{2},\nu)(\mu+\nu)^2-p_k(\mu,\nu)(2^{-k}\mu+\nu)^2\big),
$$
and \eqref{eq:ind-spectr-integ} holds at step $k+1$.

The claim \eqref{eq:sp-fo-phi_L-phi_Lkmu-2} then follows from bounding monomials $\mu^j\nu^{k-1-j}$
by $(\mu+\nu)^{k-1}$ for all $j\in \{0,\dots,k-1\}$, $\mu,\nu\ge 0$.

\medskip

\step{3} Proof of \eqref{eq:phi_L-phi_Lkmu-2}.

\noindent Note that by an a priori estimate on $D\tilde \phi_L$ and spectral calculus,
$$
1\gtrsim \, \EL{D\tilde \phi_L\cdot \aa_LD\tilde \phi_L} \, =\,\EL{\tilde \phi_L \mathbb{L}_L \tilde \phi_L}\,=\,\EL{\mathfrak{e} \mathbb{L}_L^{-1}\mathfrak{e}}\,=\,
\int_0^\infty \nu^{-1}\EL{\mathfrak{e}P(d\nu)\mathfrak{e}}.
$$
Hence,
$$
\int_{1}^\infty \frac{\mu^{2(k-1)}}{(\mu+\nu)^{2k+1}}\EL{\mathfrak{e} P_L(d\nu)\mathfrak{e}}\,
\leq \,\mu^{2(k-1)}\int_{0}^\infty \nu^{-1}\EL{\mathfrak{e} P_L(d\nu)\mathfrak{e}}\lesssim\,\mu^{2(k-1)},
$$
so that the main contribution to the RHS of \eqref{eq:phi_L-phi_Lkmu-2} comes from the spectral integral between 0 and 1, which we shall estimate using \eqref{eq:spectral-exp} --- in the spirit of \cite{Mourrat-10,Gloria-Mourrat-10}.

The fundamental theorem of calculus and Fubini's theorem imply that for all $f\in C^1((0,1])$,
\begin{eqnarray*}
\int_0^{1} f(\nu)\EL{\mathfrak{e} P_L(d\nu)\mathfrak{e}} &=& -\int_{\nu=0}^1\int_{\hat \nu=\nu}^{1}f'(\hat \nu) d\hat \nu\EL{\mathfrak{e} P_L(d\nu)\mathfrak{e}}+ f(1) \int_{\nu=0}^1 \EL{\mathfrak{e} P_L(d\nu)\mathfrak{e}}\nonumber\\
&= & -\int_{\hat \nu=0}^1f'(\hat \nu)\EL{\mathfrak{e} P_L([0,\hat \nu])\mathfrak{e}}d\hat \nu  + f(1) \EL{\mathfrak{e} P_L([0,1])\mathfrak{e}}.
\end{eqnarray*}
Since $\EL{\mathfrak{e} P_L([0,1])\mathfrak{e}}\,\leq \, \EL{\mathfrak{e}^2}\,\lesssim \, 1$, used with $f(\nu)=\frac{1 }{(\mu+\nu)^{2k+1}}$ and combined with \eqref{eq:spectral-exp}, this yields for all $k\in \N$
\begin{eqnarray*}
\int_0^{1} \frac{1}{(\mu+\nu)^{2k+1}}\EL{\mathfrak{e} P_L(d\nu)\mathfrak{e}}
&\lesssim &  \int_0^1 \frac{1}{(\mu+\nu)^{2k+2}}\EL{\mathfrak{e} P_L([0,\nu])\mathfrak{e}}d\nu+1
\\
&\stackrel{\eqref{eq:spectral-exp}}{\lesssim}&  \int_0^1 \frac{\nu^{\frac{d}{2}+1}}{(\mu+\nu)^{2k+2}}d\nu+1 \\
&\leq &  \int_0^1 \frac{1}{(\mu+\nu)^{2k+1-\frac{d}{2}}}d\nu+1 \\
&\lesssim & 
\left\{
\begin{array}{lll}
k<\frac{d}{4}&:&1,\\
k=\frac{d}{4}&:&|\log \mu|+1,\\
k>\frac{d}{4}&:&\mu^{\frac{d}{2}-2k}+1,
\end{array}
\right\},
\end{eqnarray*}
which completes the proof of \eqref{eq:phi_L-phi_Lkmu-2}.
\end{proof}
We now turn to the bound of the fourth moment, for
which we appeal to the following form of the logarithmic-Sobolev inequality (LSI) satisfied by product measures
(see in particular \cite[Lemma~4]{Marahrens-Otto-13}):
\begin{lemma}\label{lem:LSI}
Let $\mathbb P_L$ be a product measure on $\Omega_L$.
Then for all $q\ge 1$ and $\e>0$, there exists $C(q,\e)<\infty$ (independent of $L$) such
that for all $X\in L^2(\Omega_L)$,
\begin{equation}\label{eq:LSI}
\EL{|X|^{2q}}^\frac{1}{2q}\,\leq \, C(q,\e)\EL{X^2}^{\frac{1}{2}}+\e \EL{\Big(\sum_{e\in \B_L}\sup_{\aa_L(e)}\Big|\frac{\partial X}{\partial \aa_L(e)}\Big|^2\Big)^q}^{\frac{1}{2q}},
\end{equation}
where for all $e\in \B_L$, $\aa_L(e)$ denotes the $i^{th}$ entry of the diagonal matrix $\aa_L(z)$ at point $z\in \T_L$ for which $e=(z,z+\ee_i)$. 
\qed
\end{lemma}
We shall call the derivative of $X$ with respect to $\aa_L(e)$ in \eqref{eq:LSI} a vertical derivative, following the terminology of 
\cite{Gloria-Neukamm-Otto-14}.
With the help of Lemma~\ref{lem:LSI} we shall upgrade Lemma~\ref{lem:moment-2} to
\begin{lemma}\label{lem:moment-4}
For all $L\in \N$, $d\ge 2$, $\mu>0$, $k\in \N$, and $q\ge 1$,
\begin{equation}\label{eq:phi_L-phi_Lkmu-4}
\E{|\nabla \partial_{\mu}   \phi_{L,k,\mu}|^{2q}}^\frac{1}{2q}\,\lesssim\, 
\left\{
\begin{array}{lll}
k<\frac{d}{4}&:&\mu^{k-1},\\
k=\frac{d}{4}&:&\mu^{\frac{d}{4}-1}|\log \mu|^{\frac{1}{2}},\\
k>\frac{d}{4}&:&\mu^{\frac{d}{4}-1},
\end{array}
\right\}
\end{equation}
where the multiplicative constant is independent of $\mu$ and $L$.
\qed
\end{lemma}
Before we turn to the proof of Lemma~\ref{lem:moment-4} proper, we recall
the definition of periodic Green's functions, the standard quenched estimates which follow from the
De Giorgi-Nash-Moser theory, and the annealed bounds obtained by Marahrens and Otto under the 
validity of \eqref{eq:LSI}.
\begin{lemma}\label{lem:Green-annealed}
For all $L\in \N$, $\aa_L \in \Omega_L$, and $\mu\ge 0$, let $G_{L,\mu}(\cdot,\cdot;\aa_L):\T_L\times \T_L\to \R$ denote the periodic Green function, that is, for all $y\in \T_L$ the unique solution in $L^2(\T_L)$ of
$$
\mu G_{L,\mu}(x,y;\aa_L)-\nabla^*_x\cdot \aa_L(x) \nabla_x G_{L,\mu}(x,y;\aa_L)\,=\,\delta(x-y)-L^{-d}.
$$
For $d>2$ we have the pointwise bound 
\begin{equation}\label{eq:ptwise-G}
|G_{L,\mu}(x,y)|\,\lesssim \, \frac{e^{-c \sqrt{\mu} |x-y|}}{1+|x-y|^{d-2}},
\end{equation}
where $|\cdot|$ is the distance on the torus.

In addition,  if $\mathbb{P}_L$ is a product measure, then for all $d\ge 2$
and all $q\ge 1$,
\begin{eqnarray}\label{eq:annealed-nablaG}
\EL{|\nabla G_{L,\mu}(x,y)|^q}^{\frac{1}{q}}&\lesssim & \frac{e^{-c \sqrt{\mu} |x-y|}}{1+|x-y|^{d-1}},\\
\EL{|\nabla\nabla G_{L,\mu}(x,y)|^q}^{\frac{1}{q}}&\lesssim &\frac{e^{-c \sqrt{\mu} |x-y|}}{1+|x-y|^{d}},\label{eq:annealed-nabla2G}
\end{eqnarray}
where the multiplicative constant depends on $q$, next to $d$ and $\lambda$.
\qed
\end{lemma}
We refer the reader to \cite[Proof of Lemma~3.1]{Gloria-Neukamm-Otto-11} for the pointwise bound on $G_{L,\mu}$,
and to \cite[Theorem~1]{Marahrens-Otto-13} for the annealed bounds (which we have stated here
with the massive term and periodic boundary conditions, the results hold as well on $\Z^d$ and/or with $\mu=0$).

We are now in position to prove Lemma~\ref{lem:moment-4}, which is the most technical part of this article.
\begin{proof}[Proof of Lemma~\ref{lem:moment-4}]
We proceed in four steps.
In the first step we present the general strategy for $k=1$ (in particular, for $d=2$ and $d=3$
it will imply the result for all $k\in \N$).
In the second step we derive a general formula for the vertical derivative $\frac{\partial}{\aa_L(e)} (\nabla\partial_\mu\phi_{L,k,\mu})$. In the third step, we estimate the supremum of this derivative, and we conclude in the fourth and last step.

\medskip

\step{1} Proof of \eqref{eq:phi_L-phi_Lkmu-4} for $k=1$.

\medskip

\substep{1.1} Representation formula for the vertical derivative.

\noindent Differentiating with respect to $\mu$ the equation satisfied by $ \phi_{L,1,\mu}$
\begin{equation}\label{eq:prop7-1.0}
\mu  \phi_{L,1,\mu}-\nabla^*\cdot \aa_L(\xi+\nabla \phi_{L,1,\mu})\,=\,0 \quad \text{ in }\T_L
\end{equation}
yields
\begin{equation}\label{eq:prop7-1.1}
\mu \partial_\mu  \phi_{L,1,\mu}-\nabla^*\cdot \aa_L\nabla \partial_\mu\phi_{L,1,\mu}\,=\,- \phi_{L,1,\mu} \quad \text{ in }\T_L.
\end{equation}
Let $e=(z,z+\ee_i)$ for some $z\in \T_L$ and $i\in \{1,\dots,d\}$. 
The function $ \phi_{L,1,\mu}$ (and therefore  $\partial_\mu \phi_{L,1,\mu}$, in view of \eqref{eq:prop7-1.1})
is differentiable with respect to $\aa_L(e)$ (see for instance \cite[Lemma~2.4]{Gloria-Otto-09}).
Differentiating \eqref{eq:prop7-1.0} with respect to $\aa_L(e)$ yields
\begin{equation*}
\mu \frac{\partial}{\partial \aa_L(e)} \phi_{L,1,\mu}-\nabla^*\cdot \aa_L\nabla \frac{\partial}{\partial \aa_L(e)}\phi_{L,1,\mu}\,=\,\nabla^*_i \Big( (\xi +\nabla\phi_{L,1,\mu})\cdot \ee_i \delta(z-\cdot)\Big)
 \quad \text{ in }\T_L,
\end{equation*}
which we may rewrite by the Green representation formula as
\begin{eqnarray}
\frac{\partial}{\partial \aa_L(e)} \phi_{L,1,\mu}(x)&=&\sum_{y\in\T_L}G_{L,\mu}(x,y)\nabla^*_i \Big( (\xi +\nabla\phi_{L,1,\mu}(y))\cdot \ee_i \delta(z-y)\Big)\nonumber\\
&=&-\sum_{y\in\T_L}\nabla_{y_i} G_{L,\mu}(x,y) \Big( (\xi +\nabla\phi_{L,1,\mu}(y))\cdot \ee_i \delta(z-y)\Big)
\nonumber \\
&=&-\nabla_{z_i} G_{L,\mu}(x,z)(\xi \cdot \ee_i +\nabla_i\phi_{L,1,\mu}(z)).\label{eq:prop7-1.2}
\end{eqnarray}
Likewise, differentiating \eqref{eq:prop7-1.1} with respect to $\aa_L(e)$ yields  
\begin{multline*}
\mu \frac{\partial}{\partial \aa_L(e)}\partial_\mu  \phi_{L,1,\mu}-\nabla^*\cdot \aa_L\nabla \frac{\partial}{\partial \aa_L(e)}\partial_\mu\phi_{L,1,\mu}\\
=\,-\frac{\partial}{\partial \aa_L(e)} \phi_{L,1,\mu}
+\nabla^*_i \Big( \nabla_i \partial_\mu\phi_{L,1,\mu}\delta(z-\cdot)\Big)
 \quad \text{ in }\T_L,
\end{multline*}
so that we have by the Green representation formula and \eqref{eq:prop7-1.2}
\begin{eqnarray}
\lefteqn{\frac{\partial}{\partial \aa_L(e)}\partial_\mu \phi_{L,1,\mu}(x)} \label{eq:prop7-1.3b}
\\
&=&-\sum_{y\in\T_L}G_{L,\mu}(x,y)\frac{\partial}{\partial \aa_L(e)} \phi_{L,1,\mu}(y)
+\sum_{y\in\T_L}G_{L,\mu}(x,y)\nabla^*_i \Big(\nabla_i\partial_\mu\phi_{L,1,\mu}(y))\delta(z-y)\Big)\nonumber \nonumber\\
&=&(\xi\cdot \ee_i +\nabla_i\phi_{L,1,\mu}(z))\sum_{y\in\T_L}G_{L,\mu}(x,y)\nabla_{z_i} G_{L,\mu}(y,z)-\nabla_{z_i} G_{L,\mu}(x,z) \nabla_i\partial_\mu\phi_{L,1,\mu}(z),\nonumber
\end{eqnarray}
which finally yields
\begin{multline}
{\frac{\partial}{\partial \aa_L(e)}\nabla\partial_\mu \phi_{L,1,\mu}(x)}\,=\,-\nabla_x\nabla_{z_i} G_{L,\mu}(x,z)\nabla_i\partial_\mu\phi_{L,1,\mu}(z) \\
+(\xi\cdot \ee_i +\nabla_i\phi_{L,1,\mu}(z))\sum_{y\in\T_L}\nabla_x G_{L,\mu}(x,y)\nabla_{z_i} G_{L,\mu}(y,z)\label{eq:prop7-1.3}.
\end{multline}
In order to use the LSI \eqref{eq:LSI} with $X=\nabla\partial_\mu \phi_{L,1,\mu}(x)$, it remains to bound the supremum of \eqref{eq:prop7-1.3} with respect to $\aa_L(e)$.

\medskip

\substep{1.2} Supremum of the vertical derivatives and proof of
\begin{multline}
\sup_{\aa_L(e)}\Big|\frac{\partial}{\partial \aa_L(e)}\nabla\partial_\mu \phi_{L,1,\mu}(x)\Big|
\,
\lesssim\,
|\nabla\nabla G_{L,\mu}(x,z)||\nabla\partial_\mu\phi_{L,1,\mu}(z)|
\\
+(|\nabla\phi_{L,1,\mu}(z)|+1)\bigg(\sum_{y\in\T_L}|\nabla_x G_{L,\mu}(x,y)||\nabla_{z} G_{L,\mu}(y,z)|
\\
+ |\nabla\nabla G_{L,\mu}(x,z)|\sum_{y\in\T_L}|\nabla_{z} G_{L,\mu}(y,z)|^2\bigg).\label{eq:prop7-1.9}
\end{multline}
%
%
We start with the suprema of Green's functions and claim that for all $e=(z,z+\ee_i)$
and all $x,y\in\T_L$,
\begin{eqnarray}
\sup_{\aa_L(e)}|\nabla_z G_{L,\mu}(z,y)| &\lesssim & |\nabla_z G_{L,\mu}(z,y)|,\label{eq:prop7-1.4}\\
\sup_{\aa_L(e)}|\nabla_z \nabla_y G_{L,\mu}(z,y)| &\lesssim &|\nabla_z \nabla_y G_{L,\mu}(z,y)|\,\lesssim \, 1 ,\label{eq:prop7-1.5}\\
\sup_{\aa_L(e)}|\nabla_x G_{L,\mu}(x,y)| &\lesssim & |\nabla_x G_{L,\mu}(x,y)|+|\nabla_z G_{L,\mu}(z,y)||\nabla_x \nabla_z G_{L,\mu}(x,z)|.\label{eq:prop7-1.6}
\end{eqnarray}
Let $\hat \aa_L$ coincide with $\aa_L$ on $\T_L\setminus \{z\}$, let denote by $\hat G_{L,\mu}$ the
periodic Green function associated with $\hat \aa_L$, and set $\delta G:=G_{L,\mu}-\hat G_{L,\mu}$.
The function $\delta G$ satisfies for all $y\in \T_L$ the equation on $\T_L$
$$
\mu \delta G(x,y)-\nabla^*\cdot \hat \aa_L(x)\nabla \delta G(x,y)\,=\,-\nabla^*\cdot (\hat \aa_L- \aa_L)(x)\nabla G_{L,\mu}(x,y),
$$
which turns, by the Green representation formula, into
\begin{eqnarray*}
\nabla_x \delta G(x,y)&=&\sum_{y'\in \T_L}\nabla_{x} \nabla_{y'}\hat G_{L,\mu}(x,y')\cdot(\hat \aa_L(y')-\aa_L(y'))\nabla_{y'} G_{L,\mu}(y',y)\\
&=&\nabla_{x} \nabla_{z}\hat G_{L,\mu}(x,z)\cdot(\hat \aa_L(z)-\aa_L(z))\nabla_z G_{L,\mu}(z,y).
\end{eqnarray*}
Since $\sup_{\hat \aa_L\in \Omega_L}\sup_{\T_L\times \T_L} |\nabla\nabla\hat G_{L,\mu}|\lesssim 1$ (see for instance \cite[Corollary~2.3]{Gloria-Otto-09}), this implies \eqref{eq:prop7-1.4} by the triangle inequality with the choice $x=z$. This also implies \eqref{eq:prop7-1.6} by the triangle inequality
provided we prove \eqref{eq:prop7-1.5}.
Let $j\in \{1,\dots,d\}$. The function $\nabla_{y_j}\delta G(\cdot,y)$ satisfies the equation
$$
\mu \nabla_{y_j}\delta G(x,y)-\nabla^*\cdot \hat \aa_L(x)\nabla \nabla_{y_j}\delta G(x,y)\,=\,-\nabla^*\cdot (\hat \aa_L- \aa_L)(x)\nabla \nabla_{y_j}G_{L,\mu}(x,y),
$$
which turns, by the Green representation formula, into
\begin{eqnarray*}
\nabla_x \nabla_{y_j}\delta G(x,y)&=&\sum_{y'\in \T_L}\nabla_{x} \nabla_{y'}\hat G_{L,\mu}(x,y')\cdot(\hat \aa_L(y')-\aa_L(y'))\nabla_{y'} \nabla_{y_j}G_{L,\mu}(y',y)\\
&=&\nabla_{x} \nabla_{z}\hat G_{L,\mu}(x,z)\cdot(\hat \aa_L(z)-\aa_L(z))\nabla_z\nabla_{y_j} G_{L,\mu}(z,y),
\end{eqnarray*}
which proves \eqref{eq:prop7-1.5} using $\sup_{\hat\aa_L\in \Omega_L}\sup_{\T_L\times \T_L} |\nabla\nabla\hat G_{L,\mu}|\lesssim 1$ for the choice $x=z$.

\medskip

The estimate \eqref{eq:prop7-1.9} follows from \eqref{eq:prop7-1.4}---\eqref{eq:prop7-1.6}
provided we show that
\begin{eqnarray}
\sup_{\aa_L(e)}|\nabla \phi_{L,1,\mu}(x)| &\lesssim &|\nabla \phi_{L,1,\mu}(x)|+ |\nabla \phi_{L,1,\mu}(z)|,\label{eq:prop7-1.7}
\\
\sup_{\aa_L(e)}|\nabla \partial_\mu\phi_{L,1,\mu}(z)| &\lesssim & |\nabla\partial_\mu\phi_{L,1,\mu}(z)|
\nonumber\\ 
&&+(|\nabla\phi_{L,1,\mu}(z)|+1)\sum_{y\in\T_L}|\nabla_{z} G_{L,\mu}(y,z)|^2.\label{eq:prop7-1.8}
\end{eqnarray}
We start with \eqref{eq:prop7-1.7}. 
Let ${\hat\phi}_{L,1,\mu}$ be the corrector associated with $\hat \aa_L$, and set $\delta \phi:=\phi_{L,1,\mu}-{\hat\phi}_{L,1,\mu}$. The function $\delta \phi$ satisfies the equation
$$
\mu \delta \phi(x)-\nabla^*\cdot \hat \aa_L(x)\nabla \delta \phi(x)\,=\,-\nabla^*\cdot (\hat \aa_L- \aa_L)(x)\nabla \phi_{L,1,\mu}(x),
$$
which yields the a priori estimate
$$
\sum_{x\in \T_L}|\nabla \delta \phi(x)|^2\,\lesssim \, |\nabla \phi_{L,1,\mu}(z)|^2,
$$
and proves  \eqref{eq:prop7-1.7} by the triangle inequality. 

We now turn to the proof of \eqref{eq:prop7-1.8}. Formula \eqref{eq:prop7-1.3} 
for $x=z$, combined with \eqref{eq:prop7-1.7} for $x=z$ and \eqref{eq:prop7-1.4} \& \eqref{eq:prop7-1.5}, 
turns into a differential inequality for the quantity $u:\aa_L(e) \mapsto \nabla \partial_\mu\phi_{L,1,\mu}(z)$ on $[\lambda,1]$:
$$
|u'|\,\lesssim \, |u|+(|\nabla\phi_{L,1,\mu}(z)|+1)\sum_{y\in\T_L}|\nabla_{z} G_{L,\mu}(y,z)|^2,
$$
from which the desired estimate \eqref{eq:prop7-1.8} follows.
The proof of the sensitivity estimate  \eqref{eq:prop7-1.9} is complete.

\medskip

\substep{1.3} Application of the LSI \eqref{eq:LSI} to $X=\nabla\partial_\mu\phi_{L,1,\mu}(0)$.

\noindent We apply \eqref{eq:LSI} to $X=\nabla\partial_\mu\phi_{L,1,\mu}(0)$ for some general $q\ge 1$ and some $\e>0$ to be fixed later, and bound the second RHS term using \eqref{eq:prop7-1.9} and the triangle inequality:
\begin{multline}
\lefteqn{\EL{\Big(\sum_{e\in \B_L}\sup_{\aa_L(e)}\Big|\frac{\partial X}{\partial \aa_L(e)}\Big|^2\Big)^q}^{\frac{1}{2q}}}
\label{eq:prop7-1.10}
\\
\,\lesssim \, \EL{\Big(\sum_{z\in \T_L}   |\nabla\nabla G_{L,\mu}(0,z)|^2|\nabla\partial_\mu\phi_{L,1,\mu}(z)|^2\Big)^q}^{\frac{1}{2q}}+\EL{\Big(\sum_{z\in \T_L}   X_1(z)^2\Big)^q}^{\frac{1}{2q}},
\end{multline}
where we have set
\begin{multline*}
X_1(z)\,:=\,(|\nabla\phi_{L,1,\mu}(z)|+1)\\
\times \bigg(\sum_{y\in\T_L}|\nabla_x G_{L,\mu}(0,y)||\nabla_{z} G_{L,\mu}(y,z)|+ |\nabla\nabla G_{L,\mu}(0,z)|\sum_{y\in\T_L}|\nabla_{z} G_{L,\mu}(y,z)|^2\bigg).
\end{multline*}
We start with the second RHS term.
By the triangle inequality, for all $q\ge 1$,
$$
{\EL{\Big(\sum_{z\in \T_L}   X_1(z)^2\Big)^q}^{\frac{1}{q}}}\,\leq\,  \sum_{z\in \T_L}\EL{X_1(z)^{2q}}^{\frac{1}{q}}.
$$
We focus on the summand. By the H\"older and triangle inequalities,
\begin{eqnarray*}
\lefteqn{\EL{|X_1(z)|^{2q}}^{\frac{1}{2q}}}
\\
&\lesssim &(\EL{|\nabla\phi_{L,1,\mu}|^{4q}}^{\frac{1}{4q}}+1)
\sum_{y\in\T_L} \bigg(\EL{|\nabla_x G_{L,\mu}(0,y)|^{8q}}^{\frac{1}{8q}} \EL{|\nabla_{z} G_{L,\mu}(y,z)|^{8q}}^{\frac{1}{8q}}
\\
&&+ \EL{|\nabla\nabla G_{L,\mu}(0,z)|^{8q}}^{\frac{1}{8q}}\EL{|\nabla_{z} G_{L,\mu}(y,z)|^{16q}}^{\frac{1}{8q}}\bigg).
\end{eqnarray*}
We then appeal to the boundedness of the finite moments of $|\nabla\phi_{L,1,\mu}|$ (cf. \cite[Proposition~1]{Gloria-Neukamm-Otto-14}) and to the
annealed estimates \eqref{eq:annealed-nablaG} and \eqref{eq:annealed-nabla2G} on the Green functions in Lemma~\ref{lem:Green-annealed}, which yields
\begin{eqnarray*}
{\EL{|X_1(z)|^{2q}}^{\frac{1}{2q}}}
&\lesssim& \sum_{y\in\T_L}\Big( \frac{e^{-c\sqrt{\mu}|y|}}{1+|y|^{d-1}} \frac{e^{-c\sqrt{\mu}|y-z|}}{1+|y-z|^{d-1}}+\frac{e^{-c\sqrt{\mu}|z|}}{1+|z|^{d}} \frac{e^{-2c\sqrt{\mu}|y-z|}}{1+|y-z|^{2(d-1)}} \Big)
\\
&\lesssim &
\left\{
\begin{array}{lll}
d=2&:&\mu^{-\frac{1}{4}} \frac{e^{-c \sqrt{\mu} |z|}}{1+|z|^{\frac{1}{2}}}\\
d>2&:& \frac{e^{-c \sqrt{\mu} |z|}}{1+|z|^{d-2}}
\end{array}
\right\}.
\end{eqnarray*}
(For $d=2$ we have used the elementary estimate 
$\frac{e^{-c\sqrt{\mu}|y|}}{1+|y|}\lesssim \mu^{-\frac{1}{8}} \frac{e^{-c\sqrt{\mu}|y|}}{1+|y|^{1+\frac{1}{4}}}$ for $\mu\lesssim 1$.)
We thus obtain for the second RHS term of \eqref{eq:prop7-1.10}:
\begin{eqnarray}
{\EL{\Big(\sum_{z\in \T_L}   X_1(z)^2\Big)^q}^{\frac{1}{q}}}
&{\lesssim}&
\sum_{z\in \T_L}\left\{
\begin{array}{lll}
d=2&:&\mu^{-\frac{1}{2}} \frac{e^{-c \sqrt{\mu} |z|}}{1+|z|}\\
d>2&:& \frac{e^{-c \sqrt{\mu} |z|}}{1+|z|^{2(d-2)}}
\end{array}
\right\}
\nonumber\\
&\lesssim &
\left\{ 
\begin{array}{lll}
d=2&:&\mu^{-1} ,\\
d=3&:& \mu^{-\frac{1}{2}} ,\\
d=4&:&|\log \mu|,\\
d>4&:&1.
\end{array}
\right\}\label{eq:prop7-1.10a}
\end{eqnarray}

We now turn to the first RHS term of \eqref{eq:prop7-1.10} and apply H\"older's inequality to
\begin{multline*}
|\nabla\nabla G_{L,\mu}(0,z)|^2|\nabla\partial_\mu\phi_{L,1,\mu}(z)|^2\\
=\, \Big(|\nabla\nabla G_{L,\mu}(0,z)|^{\frac{2(q-1)}{q}}\Big)\Big(|\nabla\nabla G_{L,\mu}(0,z)|^{\frac{2}{q}}|\nabla\partial_\mu\phi_{L,1,\mu}(z)|^2\Big)
\end{multline*}
with exponents $(\frac{q}{q-1},q)$:
\begin{multline*}
\EL{\Big(\sum_{z\in \T_L}   |\nabla\nabla G_{L,\mu}(0,z)|^2|\nabla\partial_\mu\phi_{L,1,\mu}(z)|^2\Big)^q}^{\frac{1}{2q}}\\
\leq \, \EL{\Big(\sum_{z\in \T_L}   |\nabla\nabla G_{L,\mu}(0,z)|^{2}\Big)^{q-1}\Big(\sum_{z\in \T_L}   |\nabla\nabla G_{L,\mu}(0,z)|^2|\nabla\partial_\mu\phi_{L,1,\mu}(z)|^{2q}\Big)}^{\frac{1}{2q}}.
\end{multline*}
An elementary energy estimate on $z\mapsto \nabla_{x_i} G_{L,\mu}(0,z)$ for $i=1,\dots,d$ yields
\begin{equation}\label{eq:prop7-1.11}
\sum_{z\in \T_L}   |\nabla\nabla G_{L,\mu}(0,z)|^2\,\lesssim \,1,
\end{equation}
so that this estimate turns by stationarity into
\begin{eqnarray}
\lefteqn{\EL{\Big(\sum_{z\in \T_L}   |\nabla\nabla G_{L,\mu}(0,z)|^2|\nabla\partial_\mu\phi_{L,1,\mu}(z)|^2\Big)^q}^{\frac{1}{2q}}}\nonumber \\
&\stackrel{\eqref{eq:prop7-1.11}}{\lesssim}& \EL{\sum_{z\in \T_L}   |\nabla\nabla G_{L,\mu}(0,z)|^2|\nabla\partial_\mu\phi_{L,1,\mu}(z)|^{2q}}^{\frac{1}{2q}}\nonumber\\
&=& \EL{|\nabla\partial_\mu\phi_{L,1,\mu}(0)|^{2q}\sum_{z\in \T_L}   |\nabla\nabla G_{L,\mu}(-z,0)|^2}^{\frac{1}{2q}}\nonumber\\
&\stackrel{\eqref{eq:prop7-1.11}}{\lesssim}&\EL{|\nabla\partial_\mu\phi_{L,1,\mu}(0)|^{2q}}.\label{eq:prop7-1.12}
\end{eqnarray}

\medskip

\substep{1.4} Proof of  \eqref{eq:phi_L-phi_Lkmu-4}  for $k=1$.

\noindent 
We are now in position to conclude the proof of this step.
The combination of \eqref{eq:LSI} with \eqref{eq:prop7-1.10},  \eqref{eq:prop7-1.10a}, and \eqref{eq:prop7-1.12},  shows there exists some constant $C(q)<\infty$ such that for all $\e>0$:
\begin{multline*}
\EL{|\nabla\partial_\mu\phi_{L,1,\mu}|^{2q}}^\frac{1}{2q}\,\leq \, C(q,\e)\EL{|\nabla\partial_\mu\phi_{L,1,\mu}|^2}^{\frac{1}{2}}+\e C(q)\EL{|\nabla\partial_\mu\phi_{L,1,\mu}|^{2q}}^{\frac{1}{2q}} \\ +\e C(q)\left\{
\begin{array}{lll}
d=2&:&\mu^{-\frac{1}{2}},\\
d=3&:&\mu^{-\frac{1}{4}},\\
d=4&:&|\log \mu|^{\frac{1}{2}},\\
d>4&:&1.
\end{array}
\right\} .
\end{multline*}
We then choose $\e>0$ small enough so that we may absorb the second RHS term in the LHS.   Lemma~\ref{lem:moment-4}  for $k=1$ then follows from Lemma~\ref{lem:moment-2}.
By definition of the Richardson extrapolation, this also yields by the triangle inequality for all $k\in \N$ and all $q\ge 1$,
$$
\EL{|\nabla\partial_\mu\phi_{L,k,\mu}|^{2q}}^\frac{1}{2q}\,\lesssim \, \left\{\begin{array}{lll}
d=2&:&\mu^{-\frac{1}{2}},\\
d=3&:&\mu^{-\frac{1}{4}},\\
d=4&:&|\log \mu|^{\frac{1}{2}},\\
d>4&:&1.
\end{array}
\right\} 
$$
which, in the case $k>1$, is only optimal for $d=2,3$.
Note that combined with \eqref{eq:prop7-1.8}, the moment bounds of \cite[Proposition~1]{Gloria-Neukamm-Otto-14}, and the annealed estimate \eqref{eq:annealed-nablaG}, this also yields
$$
\EL{\sup_{\aa_L(e)}|\nabla\partial_\mu\phi_{L,k,\mu}|^{2q}}^\frac{1}{2q}\,\lesssim \, \left\{\begin{array}{lll}
d=2&:&\mu^{-\frac{1}{2}},\\
d=3&:&\mu^{-\frac{1}{4}},\\
d=4&:&|\log \mu|^{\frac{1}{2}},\\
d>4&:&1,
\end{array}
\right\} 
$$
which we will use in the sequel.
It remains to prove \eqref{eq:phi_L-phi_Lkmu-4} for $k\ge 2$. We proceed by induction. The rest of the proof
follows the same strategy as in Step~1.
Since the result is already proved for $d=2$, we shall now assume that $d>2$ so that
we do not have to use bounds on the Green function itself for $d=2$.

\medskip

\step{2} Representation formula for the vertical derivative of $\nabla \partial_\mu \phi_{L,k,\mu}$: for all $k,L\in \N$, $\mu=\mu_0> 0$, $x\in \T_L$, and $e=(z,z+\ee_i)\in \B_L$,
\begin{eqnarray}
\lefteqn{\frac{\partial}{\partial \aa_L(e)}\nabla \partial_\mu\phi_{L,k,\mu}(x)} 
\nonumber \\
&=&-\nabla_x \nabla_{z_i} G_{L,\mu}(x,z)\nabla_i \partial_\mu \phi_{L,k,\mu}(z) 
\nonumber \\
&&- \sum_{j=1}^{k-1} \bigg(\int_{\frac{\mu_{0}}{2}}^{\mu_{0}}\dots\int_{\frac{\mu_{j-1}}{2}}^{\mu_{j-1}}\mu_1\dots\mu_j \sum_{y_1\in \T_L}\dots \sum_{y_{j}\in \T_L} 
\nonumber \\
&&\qquad \qquad \times \nabla_xG_{L,\mu_1}(x,y_1)G_{L,\mu_2}(y_1,y_2)\dots G_{L,\mu_{j}}(y_{j-1},y_{j})
\nonumber \\
&&
\qquad\qquad \qquad \times \nabla_{z_i}G_{L,\mu_{j}}(y_{j},z)\nabla_i \partial_{\mu_j}\phi_{L,k-j,\mu_{j}}(z)
d\alpha_{k-j+1}(\mu_{j};\mu_{j-1})\dots d\alpha_{k}(\mu_{1};\mu_{0})
\bigg) 
\nonumber \\
&&+\int_{\frac{\mu_{0}}{2}}^{\mu_{0}}\dots\int_{\frac{\mu_{k-2}}{2}}^{\mu_{k-2}}\mu_1\dots\mu_{k-1} \sum_{y_1\in \T_L}\dots \sum_{y_{k}\in \T_L} 
\nonumber \\
&&\qquad \qquad \times \nabla_xG_{L,\mu_1}(x,y_1)G_{L,\mu_2}(y_1,y_2)\dots G_{L,\mu_{k-1}}(y_{k-2},y_{k-1})
\nonumber \\
&&
\qquad\qquad \qquad \times G_{L,\mu_{k-1}}(y_{k-1},y_{k})\nabla_{z_i}G_{L,\mu_{k-1}}(y_{k},z) (\xi \cdot \ee_i+\nabla_i \phi_{L,1,\mu_{k-1}}(z))
\nonumber \\
&&\qquad\qquad \qquad\qquad \times d\alpha_{2}(\mu_{k-1};\mu_{k-2})\dots d\alpha_{k}(\mu_{1};\mu_{0}),
\label{eq:prop7-2.1}
\end{eqnarray}
where $d\alpha_{j}(\cdot;\mu)$ is a positive measure on the interval 
$(\frac{\mu}{2},\mu)$ of total mass bounded by $2$.
This formula directly follows from the corresponding formula for $\frac{\partial}{\partial \aa_L(e)} \partial_\mu\phi_{L,k,\mu}(x)$:
\begin{eqnarray}
\lefteqn{\frac{\partial}{\partial \aa_L(e)}\partial_\mu\phi_{L,k,\mu}(x)} 
\nonumber \\
&=&- \nabla_{z_i} G_{L,\mu}(x,z)\nabla_i \partial_\mu \phi_{L,k,\mu}(z) 
\nonumber \\
&&- \sum_{j=1}^{k-1} \bigg(\int_{\frac{\mu_{0}}{2}}^{\mu_{0}}\dots\int_{\frac{\mu_{j-1}}{2}}^{\mu_{j-1}}\mu_1\dots\mu_j \sum_{y_1\in \T_L}\dots \sum_{y_{j}\in \T_L} 
\nonumber \\
&&\qquad \qquad \times G_{L,\mu_1}(x,y_1)G_{L,\mu_2}(y_1,y_2)\dots G_{L,\mu_{j}}(y_{j-1},y_{j})
\nonumber \\
&&
\qquad\qquad \qquad \times \nabla_{z_i}G_{L,\mu_{j}}(y_{j},z)\nabla_i \partial_{\mu_j}\phi_{L,k-j,\mu_{j}}(z)
d\alpha_{k-j+1}(\mu_{j};\mu_{j-1})\dots d\alpha_{k}(\mu_{1};\mu_{0})
\bigg) 
\nonumber \\
&&+\int_{\frac{\mu_{0}}{2}}^{\mu_{0}}\dots\int_{\frac{\mu_{k-2}}{2}}^{\mu_{k-2}}\mu_1\dots\mu_{k-1} \sum_{y_1\in \T_L}\dots \sum_{y_{k}\in \T_L} 
\nonumber \\
&&\qquad \qquad \times G_{L,\mu_1}(x,y_1)G_{L,\mu_2}(y_1,y_2)\dots G_{L,\mu_{k-1}}(y_{k-2},y_{k-1})
\nonumber \\
&&
\qquad\qquad \qquad \times G_{L,\mu_{k-1}}(y_{k-1},y_{k})\nabla_{z_i}G_{L,\mu_{k-1}}(y_{k},z) (\xi\cdot \ee_i+\nabla_i \phi_{L,1,\mu_{k-1}}(z))
\nonumber \\
&&\qquad\qquad \qquad\qquad \times d\alpha_{2}(\mu_{k-1};\mu_{k-2})\dots d\alpha_{k}(\mu_{1};\mu_{0}).
\label{eq:prop7-2.4}
\end{eqnarray}
\noindent By induction, it holds that
\begin{equation*}
\mu \phi_{L,k+1,\mu}-\nabla^*\cdot \aa_L(\xi+\nabla \phi_{L,k+1,\mu})\,=\,\mu \phi_{L,k,\mu}.
\end{equation*}
Differentiating this equation with respect to $\mu$ yields
\begin{equation}\label{eq:prop7-2.2}
\mu \partial_\mu\phi_{L,k+1,\mu}-\nabla^*\cdot \aa_L\nabla \partial_\mu\phi_{L,k+1,\mu}\,=\,\mu \partial_\mu\phi_{L,k,\mu}+(\phi_{L,k,\mu}-\phi_{L,k+1,\mu}).
\end{equation}
By definition \eqref{eq:Richardson1} of the Richardson extrapolation we rewrite the second RHS term as
\begin{eqnarray*}
\phi_{L,k,\mu}-\phi_{L,k+1,\mu}&=&\frac{1}{2^k-1}((2^k-1)\phi_{L,k,\mu}-2^k \phi_{L,k,\frac{\mu}{2}}+\phi_{L,k,\mu}) \\
&=&\frac{2^k}{2^k-1}(\phi_{L,k,\mu}-\phi_{L,k,\frac{\mu}{2}})
\\
&=&\frac{2^k}{2^k-1} \int_{\frac{\mu}{2}}^\mu \partial_{\mu_1}\phi_{L,k,\mu_1}d\mu_1,
\end{eqnarray*}
so that \eqref{eq:prop7-2.2} turns into
\begin{equation}\label{eq:prop7-2.3}
\mu \partial_\mu\phi_{L,k+1,\mu}-\nabla^*\cdot \aa_L\nabla \partial_\mu\phi_{L,k+1,\mu}\,=\,
\mu \int_{\frac{\mu}{2}}^\mu \partial_{\mu_1}\phi_{L,k,\mu_1}d\alpha_k(\mu_1;\mu),
\end{equation}
where $d\alpha_k(\mu_1;\mu):=\frac{2^k}{\mu(2^k-1)} d\mu_1+\delta(\mu_1-\mu)$, which satisfies
as claimed $\int_{\frac{\mu}{2}}^\mu d\alpha_k(\mu_1;\mu)\,=\,\frac{2^{k-1}}{2^k-1}+1\le 2$.
Next we differentiate \eqref{eq:prop7-2.3} with respect to $\aa_L(e)$ and obtain
\begin{multline*}
\mu \frac{\partial}{\aa_L(e)}\partial_\mu\phi_{L,k+1,\mu}-\nabla^*\cdot \aa_L\nabla \frac{\partial}{\aa_L(e)}\partial_\mu\phi_{L,k+1,\mu}\,=\,\mu
\int_{\frac{\mu}{2}}^\mu \frac{\partial}{\aa_L(e)}\partial_{\mu_1}\phi_{L,k,\mu_1}d\alpha_k(\mu_1;\mu)
\\
+\nabla^*_i \Big( \nabla_i\partial_\mu \phi_{L,k+1,\mu} \delta(z-\cdot)\Big)
 \quad \text{ in }\T_L.
\end{multline*}
This yields the Green representation formula
\begin{multline*}
\frac{\partial}{\aa_L(e)}\partial_\mu\phi_{L,k+1,\mu}(x)\,=\,-\nabla_{z_i}G_{L,\mu}(x,z)\nabla_i \partial_\mu \phi_{L,k+1,\mu}(z) \\
+\sum_{y \in \T_L} G_{L,\mu}(x,y)\mu\int_{\frac{\mu}{2}}^\mu \frac{\partial}{\aa_L(e)}\partial_{\mu_1}\phi_{L,k,\mu_1}(y)
d\alpha_k(\mu_1;\mu),
\end{multline*}
from which \eqref{eq:prop7-2.1} follows by induction starting with \eqref{eq:prop7-1.3b} for $k=1$.

\medskip

\step{3} Supremum of the vertical derivative of $\nabla \partial_\mu  \phi_{L,k,\mu}$: for all $k,L\in \N$, $\mu=\mu_0> 0$,
$x\in \B_L$, 
and $e=(z,z+\ee_i)\in \B_L$,
\begin{eqnarray}
\lefteqn{\sup_{\aa_L(e)}\Big|\frac{\partial}{\partial \aa_L(e)}\nabla\partial_\mu\phi_{L,k,\mu}(x)\Big|} 
\nonumber\\
&=&|\nabla_x \nabla_{z_i} G_{L,\mu}(x,z)||\nabla \partial_\mu \phi_{L,k,\mu}(z)| 
\nonumber\\
&&+ \sum_{j=1}^{k-1} \bigg(\int_{\frac{\mu_{0}}{2}}^{\mu_{0}}\dots\int_{\frac{\mu_{j-1}}{2}}^{\mu_{j-1}}\mu_1\dots\mu_j \sum_{y_1\in \T_L}\dots \sum_{y_{j}\in \T_L}
\nonumber\\
&&\quad \qquad \times \Big(|\nabla_xG_{L,\mu_1}(x,y_1)|+|\nabla_z G_{L,\mu_1}(y_1,z)|\big(|\nabla\nabla G_{L,\mu_1}(x,z)|+|\nabla\nabla G_{L,\mu}(x,z)|\big)\Big)
\nonumber\\
&&
\quad\qquad\qquad \times g_\mu(|y_1-y_2|)\dots g_{\mu}(|y_{j-1}-y_{j}|) |\nabla_{z}G_{L,\mu_{j}}(y_{j},z)| \sup_{\aa_L(e)}|\nabla_i \partial_{\mu_j}\phi_{L,k-j,\mu_{j}}(z)|
\nonumber\\
&&
\quad\qquad\qquad \qquad\times d\alpha_{k-j+1}(\mu_{j};\mu_{j-1})\dots d\alpha_{k}(\mu_{1};\mu_{0})
\bigg) 
\nonumber\\
&&+\int_{\frac{\mu_{0}}{2}}^{\mu_{0}}\dots\int_{\frac{\mu_{k-2}}{2}}^{\mu_{k-2}}\mu_1\dots\mu_{k-1} \sum_{y_1\in \T_L}\dots \sum_{y_{k}\in \T_L} 
\nonumber\\
&&\qquad  \times \Big(|\nabla_xG_{L,\mu_1}(x,y_1)|+|\nabla_z G_{L,\mu_1}(y_1,z)|\big(|\nabla\nabla G_{L,\mu_1}(x,z)|+|\nabla\nabla G_{L,\mu}(x,z)|\big)\Big)
\nonumber\\
&&
\qquad\qquad \times g_{\mu}(|y_1-y_2|)\dots g_{\mu}(|y_{k-1}-y_{k}|)|\nabla_{z}G_{L,\mu_{k-1}}(y_{k},z)| (1+|\nabla \phi_{L,1,\mu_{k-1}}(z)|)
\nonumber\\
&&\qquad\qquad \qquad\qquad \times d\alpha_{2}(\mu_{k-1};\mu_{k-2})\dots d\alpha_{k}(\mu_{1};\mu_{0}),
\label{eq:prop7-3.1}
\end{eqnarray}
where $g_\mu:\R^+\to \R^+$ is defined for $d>2$ by
\begin{equation}\label{eq:gmu}
g_\mu(t)\,:=\,\frac{e^{-c \sqrt{\mu} t}}{1+t^{d-2}}
\end{equation}
for some $c>0$ depending
only on $k$, $\lambda$ and $d$.

\noindent Set
\begin{eqnarray*}
\alpha&:=&\sum_{j=1}^{k-1} \bigg(\int_{\frac{\mu_{0}}{2}}^{\mu_{0}}\dots\int_{\frac{\mu_{j-1}}{2}}^{\mu_{j-1}}\mu_1\dots\mu_j \sum_{y_1\in \T_L}\dots \sum_{y_{j}\in \T_L}
\nonumber\\
&& \qquad \times \Big(|\nabla_xG_{L,\mu_1}(x,y_1)|+|\nabla_z G_{L,\mu_1}(y_1,z)|\big(|\nabla\nabla G_{L,\mu_1}(x,z)|+|\nabla\nabla G_{L,\mu}(x,z)|\big)\Big)
\nonumber\\
&&
\quad\qquad \times g_\mu(|y_1-y_2|)\dots g_{\mu}(|y_{j-1}-y_{j}|) |\nabla_{z}G_{L,\mu_{j}}(y_{j},z)| \sup_{\aa_L(e)}|\nabla_i \partial_{\mu_j}\phi_{L,k-j,\mu_{j}}(z)|
\nonumber\\
&&
\quad\qquad \qquad\times d\alpha_{k-j+1}(\mu_{j};\mu_{j-1})\dots d\alpha_{k}(\mu_{1};\mu_{0})
\bigg) 
\nonumber\\
\beta&:=&\int_{\frac{\mu_{0}}{2}}^{\mu_{0}}\dots\int_{\frac{\mu_{k-2}}{2}}^{\mu_{k-2}}\mu_1\dots\mu_{k-1} \sum_{y_1\in \T_L}\dots \sum_{y_{k}\in \T_L} 
\nonumber\\
&&\quad  \times \Big(|\nabla_xG_{L,\mu_1}(x,y_1)|+|\nabla_z G_{L,\mu_1}(y_1,z)|\big(|\nabla\nabla G_{L,\mu_1}(x,z)|+|\nabla\nabla G_{L,\mu}(x,z)|\big)\Big)
\nonumber\\
&&
\quad\qquad \times g_{\mu}(|y_1-y_2|)\dots g_{\mu}(|y_{k-1}-y_{k}|)|\nabla_{z}G_{L,\mu_{k-1}}(y_{k},z)| (1+|\nabla \phi_{L,1,\mu_{k-1}}(z)|)
\nonumber\\
&&\quad\qquad \qquad\qquad \times d\alpha_{2}(\mu_{k-1};\mu_{k-2})\dots d\alpha_{k}(\mu_{1};\mu_{0}),
\end{eqnarray*}
so that by  \eqref{eq:prop7-1.4}---\eqref{eq:prop7-1.7}, we have 
$\sup_{\aa_L(e)}\alpha \lesssim \alpha$ and $\sup_{\aa_L(e)}\beta\lesssim \beta$.
Then, by \eqref{eq:ptwise-G} in Lemma~\ref{lem:Green-annealed}, 
\eqref{eq:prop7-2.1} turns into
\begin{equation}\label{eq:prop7-3.2}
\Big|\frac{\partial}{\partial \aa_L(e)}\nabla\partial_\mu\phi_{L,k,\mu}(x)\Big|
\,\lesssim \, |\nabla_x \nabla_{z_i} G_{L,\mu}(x,z)||\nabla_i \partial_\mu \phi_{L,k,\mu}(z)| +\alpha +\beta.
\end{equation}
For $x=z$, using  \eqref{eq:prop7-1.5} in the form of $\sup |\nabla \nabla G| \lesssim 1$, this yields a differential inequality for $\aa_L(e)\mapsto \nabla\partial_\mu\phi_{L,k,\mu}(z)$, from which we infer that
\begin{equation}\label{eq:prop7-3.3}
\sup_{\aa_L(e)}|\nabla\partial_\mu\phi_{L,k,\mu}(z)|\,\lesssim \,|\nabla\partial_\mu\phi_{L,k,\mu}(z)|+
\alpha+\beta.
\end{equation}
The combination of \eqref{eq:prop7-3.2} and \eqref{eq:prop7-3.3} yields the desired estimate \eqref{eq:prop7-3.1}.

\medskip

\step{4} Proof of \eqref{eq:phi_L-phi_Lkmu-4} by induction.

\noindent The induction assumption at step $k\in \N$ reads: for all $q\ge 1$,
\begin{equation}\label{eq:prop7-4.1}
\EL{\sup_{\aa_L(0)}|\nabla \partial_\mu  \phi_{L,k,\mu}(0)|^{2q}}^{\frac{1}{2q}}
\,\lesssim \, 
\left\{
\begin{array}{lll}
k<\frac{d}{4}&:&\mu^{k-1},\\
k=\frac{d}{4}&:&\mu^{\frac{d}{4}-1}|\log \mu|^{\frac{1}{2}},\\
k>\frac{d}{4}&:&\mu^{\frac{d}{4}-1},
\end{array}
\right\}
\end{equation}
which directly implies the claim by discarding the supremum in the expectation.

\noindent For $k=1$, \eqref{eq:prop7-4.1} is a consequence of \eqref{eq:prop7-1.9} and \eqref{eq:phi_L-phi_Lkmu-4} (which we proved in Step~1 for $k=1$), since for all $q\ge 1$, $\sup_{\aa_L\in \Omega_L}|\nabla\nabla G_{L,\mu}(0,0)|\,\lesssim\,1$, $\EL{|\nabla \phi_{L,1,\mu}|^q}^\frac{1}{q}\,\lesssim\,1$ uniformly with respect to $\mu\ge 0$, and 
$$
\EL{\Big(\sum_{y\in\T_L}|\nabla G_{L,\mu}(0,y)|^2\Big)^q}^{\frac{1}{q}}\,\lesssim\, 
\left\{
\begin{array}{lll}
d=2&:&|\log \mu|,\\
d>2&:&1.
\end{array}
\right\}
$$

Assume now that \eqref{eq:prop7-4.1} holds at step~$k\in \N$.
From the logarithmic-Sobolev inequality in the form of \eqref{eq:LSI} and the sensitivity estimate \eqref{eq:prop7-3.1},
we learn that $\EL{|\nabla \phi_{L,k+1,\mu}|^{2q}}^{\frac{1}{2q}}$ is bounded by the sum of $k+3$ terms:
the second moment $\EL{|\nabla \phi_{L,k+1,\mu}|^{2}}^{\frac{1}{2}}$ which is controlled by Lemma~\ref{lem:moment-2},
the nonlinear term 
$$C(q)\e \EL{\Big(\sum_{z\in \T_L}|\nabla \nabla G_{L,\mu}(0,z)|^2 |\nabla \phi_{L,k+1,\mu}(z)|^{2}\Big)^q}^{\frac{1}{2q}},$$
which we absorb in the LHS for $\e$ small enough (arguing as for \eqref{eq:prop7-1.12}), and $k+1$ linear terms.
We start with the estimate of the last linear term, which involves $\nabla \phi_{L,1,\mu_{k}}$.
In addition to $g_\mu$ (cf. \eqref{eq:gmu}), we define $h_\mu,\gamma_\mu:\R^+\to \R^+$ by
\begin{equation*}
h_\mu(t)\,:=\,\frac{e^{-c\sqrt{\mu}t}}{1+t^{d-1}},\qquad \gamma_\mu(t)\,:=\,\frac{e^{-c\sqrt{\mu}t}}{1+t^{d}}.
\end{equation*}
Since all the finite moments of $\nabla \phi_{L,1,\mu_{k}}$ are bounded and the measures $d\alpha_{k-j}$ have
mass of order one, we have by H\"older's inequality, 
and the annealed estimates \eqref{eq:annealed-nablaG} and \eqref{eq:annealed-nabla2G}:
\begin{eqnarray*}
\lefteqn{\mathbb{E}_L\bigg[\bigg(
\sum_{z\in \T_L}\Big(
\int_{\frac{\mu_{0}}{2}}^{\mu_{0}}\dots\int_{\frac{\mu_{k-1}}{2}}^{\mu_{k-1}}\mu_1\dots\mu_{k} \sum_{y_1\in \T_L}\dots \sum_{y_{k+1}\in \T_L} }
\nonumber\\
&& \times \Big(|\nabla_xG_{L,\mu_1}(x,y_1)|+|\nabla_z G_{L,\mu_1}(y_1,z)|\big(|\nabla\nabla G_{L,\mu_1}(x,z)|+|\nabla\nabla G_{L,\mu}(x,z)|\big)\Big)
\nonumber\\
&&
\qquad \times g_{\mu}(|y_1-y_2|)\dots g_{\mu}(|y_{k}-y_{k+1}|)|\nabla_{z}G_{L,\mu_{k}}(y_{k+1},z)| (1+|\nabla \phi_{L,1,\mu_{k}}(z)|)
\nonumber\\
&&\qquad \times d\alpha_{2}(\mu_{k};\mu_{k-1})\dots d\alpha_{k+1}(\mu_{1};\mu_{0})\Big)^2\bigg)^q
\bigg]^\frac{1}{2q}
\\
&\lesssim & \mu^{k-1}\bigg(\sum_{z,y_1,\dots,y_{k+1},y_1'\dots,y_{k+1}'\in \T_L} (h_{\mu}(|y_1|)+h_{\mu}(|y_1-z|) \gamma_{\mu}(|z|))\\
&&\quad \times(h_{\mu}(|y_1'|)+h_{\mu}(|y_1'-z|)\gamma_{\mu}(|z|)) g_{\mu}(y_1-y_2)\dots g_{\mu}(y_{k-1}-y_{k})g_{\mu}(y_{k}-y_{k+1})\\
&&\qquad\times g_{\mu}(|y_1'-y_2'|)\dots g_{\mu}(|y_{k-1}'-y_{k}'|)g_{\mu}(|y_{k}'-y_{k+1}'|)h_{\mu}(|y_{k+1}-z|)
h_{\mu}(|y_{k+1}'-z|)\bigg)^{\frac{1}{2}}.
\end{eqnarray*}
Since the functions $g_\mu$, $h_\mu$ and $\gamma_\mu$ are bounded pointwise for $d>2$, we may directly estimate this
$(2k+3)$-ple sum (by comparison to the corresponding integrals), which yields the desired RHS of \eqref{eq:prop7-4.1}.

We now estimate the $k$ remaining linear terms, which involve $\{\nabla  \partial_\mu \phi_{L,j,\mu}\}_{j=1,\dots,k}$. 
We proceed as above and use in addition the induction assumption up to step~$k$ in the suboptimal form
for $d>2$, for all $1\le j\le k$ and $q\ge 1$:
$$
\EL{\sup_{\aa_L(0)}|\nabla \partial_\mu \phi_{L,j,\mu}(0)|^{2q}}^\frac{1}{2q}\,\lesssim\, \mu^{-\frac{1}{4}},
$$
which follows from bounding the RHS of \eqref{eq:prop7-4.1} for $d=3$.
This yields for all $1\le j\le k$,
\begin{eqnarray*}
\lefteqn{\mathbb{E}_L\bigg[\bigg(\sum_{z\in \T_L} \bigg(\int_{\frac{\mu_{0}}{2}}^{\mu_{0}}\dots\int_{\frac{\mu_{j-1}}{2}}^{\mu_{j-1}}\mu_1\dots\mu_j \sum_{y_1\in \T_L}\dots \sum_{y_{j}\in \T_L}}
\nonumber\\
&& \qquad \times \Big(|\nabla_xG_{L,\mu_1}(x,y_1)|+|\nabla_z G_{L,\mu_1}(y_1,z)|\big(|\nabla\nabla G_{L,\mu_1}(x,z)|+|\nabla\nabla G_{L,\mu}(x,z)|\big)\Big)
\nonumber\\
&&
\qquad\qquad \times g_\mu(|y_1-y_2|)\dots g_{\mu}(|y_{j-1}-y_{j}|) |\nabla_{z}G_{L,\mu_{j}}(y_{j},z)| \sup_{\aa_L(e)}|\nabla_i \partial_{\mu_j}\phi_{L,k-j,\mu_{j}}(z)|
\nonumber\\
&&
\qquad\qquad \qquad\times d\alpha_{k-j}(\mu_{j};\mu_{j-1})\dots d\alpha_{k+1}(\mu_{1};\mu_{0})
\bigg) ^2\bigg)^q\bigg]^{\frac{1}{2q}}
\\
&\lesssim & \mu^{\frac{d}{4}-\frac{1}{4}} \bigg(\mu^{2j-\frac{d}{2}}\sum_{z,y_1,\dots,y_j,y_1',\dots,y_j'\in \T_L} (h_\mu(|y_1|)+h_{\mu}(|y_1-z|)\gamma_\mu(|x-z|))
\nonumber\\
&&\qquad \times(h_\mu(|y_1'|)+h_{\mu}(|y_1'-z|)\gamma_\mu(|z|))
g_\mu(|y_1-y_2|)\dots g_{\mu}(|y_{j-1}-y_{j}|) h_\mu(|y_{j}-z|)\bigg)^{\frac{1}{2}} \\
&\lesssim & \mu^{\frac{d}{4}-\frac{1}{4}},
\end{eqnarray*}
by a direct comparison of the $(2j+1)$-ple sum for $d>2$ to integrals.
This estimate is of higher order than the RHS of \eqref{eq:prop7-4.1} and holds for all $1\le j\le k$.
This proves that 
\begin{equation*}
\EL{|\nabla \partial_\mu  \phi_{L,k,\mu}(0)|^{2q}}^{\frac{1}{2q}}
\,\lesssim \, 
\left\{
\begin{array}{lll}
k<\frac{d}{4}&:&\mu^{k-1},\\
k=\frac{d}{4}&:&\mu^{\frac{d}{4}-1}|\log \mu|^{\frac{1}{2}},\\
k>\frac{d}{4}&:&\mu^{\frac{d}{4}-1}.
\end{array}
\right\}
\end{equation*}
To prove the same bound on the supremum \eqref{eq:prop7-4.1}, we appeal to \eqref{eq:prop7-3.3}, and bound
the terms $\alpha$ and $\beta$ as above using in addition the induction assumption.

\end{proof}

\subsection{Proof of Proposition~\ref{prop:var-resc-mu}}

We start by proving that \eqref{eq:formal-sigma-mu} and the associated Richardson extrapolation variants are well-defined.
\begin{lemma}\label{lem:conv-series}
For all $k\in \N$ and $\mu>0$, let $\phi_{k,\mu}$ be  as in Definition~\ref{def:Richardson}.
Then, $\sigma_{k,\mu}^2\in \R^+$ is well-defined as the following sum:
\begin{equation}\label{eq:conv-series}
\sigma_{k,\mu}^2 \,=\,\sum_{x\in \Z^d} \cov{(\xi+\nabla \phi_{k,\mu})\cdot \aa(\xi+\nabla \phi_{k,\mu})(x)}{(\xi+\nabla \phi_{k,\mu})\cdot \aa(\xi+\nabla \phi_{k,\mu})(0)}.
\end{equation}
\end{lemma}
In order to prove Lemma~\ref{lem:conv-series} we shall appeal to the following
covariance estimate valid for product measures (cf. \cite[Lemma~3]{Gloria-Otto-09b}):
\begin{lemma}\label{lem:covar}
Let $\mathbb P_L$ be a product measure on $\Omega_L$ for $L\in \N\cup\{+\infty\}$ (with $\Omega_L=\Omega$ for $L=+\infty$).
Then for all $X,Y\in L^2(\Omega_L)$,
\begin{equation}\label{eq:covar}
\covL{X}{Y}\,\lesssim \, \sum_{e\in \B_L}\EL{\sup_{\aa_L(e)}\Big|\frac{\partial X}{\partial \aa_L(e)}\Big|^2}^{\frac{1}{2}}
\EL{\sup_{\aa_L(e)}\Big|\frac{\partial Y}{\partial \aa_L(e)}\Big|^2}^{\frac{1}{2}},
\end{equation}
where for all $e\in \B_L$, $\aa_L(e)$ denotes the $i^{th}$ entry of the diagonal matrix $\aa_L(z)$ at point $z\in\Z^d$ for which $e=(z,z+\ee_i)$. 
\qed
\end{lemma}
\begin{proof}[Proof of Lemma~\ref{lem:conv-series}]
We split the proof into three steps.
We first estimate vertical derivatives for $k=1$, then prove the summability of the series by applying the covariance estimate of Lemma~\ref{lem:covar} and appealing to the annealed estimates of Lemma~\ref{lem:Green-annealed}.
We then conclude in the last step for $k>1$.

\medskip

\step{1} Vertical derivative of the energy density for $k=1$ and proof of 
\begin{multline}\label{eq:vert-der-ener}
\sup_{\aa(e)}\Big|\frac{\partial \ener_\mu(x)}{\partial \aa(e)}\Big|\,\lesssim \,(1+|\nabla \phi_{\mu}(z)|)^2\delta(z-x)
\\+| \nabla_x\nabla_{z_i} G_\mu(x,z)|(1+|\nabla \phi_{\mu}(x)|^2+|\nabla \phi_{\mu}(z)|^2),
\end{multline}
where $\ener_\mu(x):=(\xi+\nabla \phi_{\mu})\cdot \aa(\xi+\nabla \phi_{\mu})(x)$, $e=(z,z+\ee_i)$ for some $i\in \{1,\dots, d\}$, and $G_\mu$ is the massive Green function on $\Z^d$ associated with the operator $\mu-\nabla^*\cdot\aa \nabla$.

\noindent By the Leibniz rule, 
$$
\frac{\partial \ener_\mu(x)}{\partial \aa(e)}\,=\,(\xi+\nabla \phi_{\mu}(x))\cdot \frac{\partial \aa(x)}{\partial \aa(e)}(\xi+\nabla \phi_{\mu}(x))+2\nabla\frac{\partial \phi_{\mu}(x)}{\partial \aa(e)}\cdot \aa(x)(\xi+\nabla \phi_{\mu}(x)),
$$
so that \eqref{eq:prop7-1.2} (in its whole space version, the proof of which is identical) yields
$$
\frac{\partial \ener_\mu(x)}{\partial \aa(e)}\,=\,(\xi\cdot \ee_i+\nabla_i \phi_{\mu}(z))^2\delta(z-x)-2\nabla_i \phi_\mu(z) \nabla_x\nabla_{z_i} G_\mu(x,z)\cdot \aa(x)(\xi+\nabla \phi_{\mu}(x)).
$$
The desired estimate then follows from taking the supremum over $\aa(e)$ on the mixed second gradient of
the Green function and on the gradient of the corrector, as we already did in \eqref{eq:prop7-1.5} and \eqref{eq:prop7-1.7} on the $L$-torus (the proofs are identical).

\medskip

\step{2} Proof of the summability of the RHS of \eqref{eq:conv-series} for $k=1$.

\noindent
Let $x\in \Z^d$. The covariance estimate \eqref{eq:covar}  on $\Omega$ yields
\begin{eqnarray*}
\cov{\ener_\mu(x)}{\ener_\mu(0)}&\lesssim & \sum_{e\in \B}\E{\sup_{\aa(e)}\Big|\frac{\partial \ener_\mu(x)}{\partial \aa(e)}\Big|^2}^{\frac{1}{2}}
\E{\sup_{\aa(e)}\Big|\frac{\partial \ener_\mu(0)}{\partial \aa(e)}\Big|^2}^{\frac{1}{2}}.
\end{eqnarray*}
We estimate each term using \eqref{eq:vert-der-ener},  the boundedness of the fourth moment of the 
gradient of the corrector (cf. \cite[Proposition~1]{Gloria-Neukamm-Otto-14}) and the annealed estimate \eqref{eq:annealed-nablaG} (also valid on the whole space with the massive term),
so that by Cauchy-Schwarz' inequality and stationarity of $\nabla  \phi_\mu$,
\begin{equation}\label{eq:covar-x-0-ener.0}
\E{\sup_{\aa(e)}\Big|\frac{\partial \ener_\mu(x)}{\partial \aa(e)}\Big|^2}^{\frac{1}{2}}
\,\lesssim\, \delta(z-x)+\frac{e^{-c\sqrt{\mu}|z-x|}}{1+|z-x|^d}.
\end{equation}
We may then estimate the above covariance as follows:
\begin{equation}\label{eq:covar-x-0-ener}
|\cov{\ener_\mu(x)}{\ener_\mu(0)}|\,\lesssim \,\sum_{z\in \Z^d} (\delta(z-x)+\frac{e^{-c\sqrt{\mu}|z-x|}}{1+|z-x|^d})( \delta(z)+\frac{e^{-c\sqrt{\mu}|z|}}{1+|z|^d})
\,\lesssim \, \frac{e^{-c\sqrt{\mu}|x|}}{1+|x|^d}.
\end{equation}
Summing \eqref{eq:covar-x-0-ener} over $x\in \Z^d$ finally yields
\begin{equation*}
\sum_{x\in \Z^d}|\cov{\ener_\mu(x)}{\ener_\mu(0)}|\,\lesssim \,\sum_{x\in \Z^d} \frac{e^{-c\sqrt{\mu}|x|}}{1+|x|^d} \, \lesssim \, 1+|\log \mu|,
\end{equation*}
that is, the claimed summability.

\medskip

\step{3} Reduction to the case $k=1$.

\noindent By definition of the Richardson extrapolation, 
$\nabla \phi_{k,\mu}\,=\,\sum_{i=0}^{k-1} c_{j,k} \nabla \phi_{k,\frac{\mu}{2^j}}$ for some 
coefficients $c_{j,k}$.
The result for $k>1$ thus follows from the corresponding result for $k=1$ by the triangle inequality.
\end{proof}

We are in position to prove Proposition~\ref{prop:var-resc-mu}.
For notational convenience we center the periodic cell at $0$ in the following.
\begin{proof}[Proof of Proposition~\ref{prop:var-resc-mu}]
As for the proof of Lemma~\ref{lem:conv-series}, by definition of extrapolation and by the triangle inequality it is enough
to prove the claim for $k=1$, so that we only consider $k=1$ in the proof and we drop the subscript $k$
in our notation.
We split the proof into four steps. We start with a reformulation of $\sigma_\mu^2-\sigma_{L,\mu}^2$
as the sum of three terms, which we estimate one by one.

\medskip

\step{1} Reduction of the claim.

\noindent We first argue that
\begin{equation}\label{eq:refo-sigmaLmu}
\sigma_{L,\mu}^2\,:=\,L^{d}\var{a_{L,\mu}}\,=\,\sum_{x\in [-\lceil \frac{L}{2}\rceil,\frac{L}{2})^d\cap \Z^d} \cov{\ener_{L,\mu}(x)}{\ener_{L,\mu}(0)},
\end{equation}
where $\ener_{L,\mu}(x)\,=\,(\xi+\nabla \phi_{L,\mu})\cdot\aa_L(\xi+\nabla \phi_{L,\mu})(x)$
and $\lceil t \rceil$ denotes the smallest integer larger or equal to $t$.
Indeed, by stationarity,
\begin{multline*}
\sigma_{L,\mu}^2\,=\,L^{-d} \sum_{x\in \T_L}\sum_{y\in \T_L} \cov{\ener_{L,\mu}(x)}{\ener_{L,\mu}(y)}
\\
\,=\,L^{-d} \sum_{x\in \T_L}\sum_{y\in \T_L} \cov{\ener_{L,\mu}(x-y)}{\ener_{L,\mu}(0)}\,=\,\sum_{x\in \T_L}\cov{\ener_{L,\mu}(x)}{\ener_{L,\mu}(0)},
\end{multline*}
so that \eqref{eq:refo-sigmaLmu} follows by choosing $[-\lceil \frac{L}{2}\rceil,\frac{L}{2})^d\cap \Z^d$ as a representation of $\T_L$.

In particular, \eqref{eq:refo-sigmaLmu} allows one to rewrite the difference $\sigma^2_\mu-\sigma_{L,\mu}^2$
as the sum of three terms
$$
\sigma^2_\mu-\sigma_{L,\mu}^2\,=\,I_1(L,\mu)+I_2(L,\mu)+I_3(L,\mu),
$$
where
\begin{eqnarray*}
I_1(L,\mu)&:=&\sum_{x\in  [-\lceil \frac{L}{4}\rceil,\frac{L}{4})^d\cap \Z^d} \Big( \cov{\ener_{\mu}(x)}{\ener_{\mu}(0)}-
\cov{\ener_{L,\mu}(x)}{\ener_{L,\mu}(0)}\Big),\label{eq:def-I1}\\
I_2(L,\mu)&:=&\sum_{x\in \Z^d \setminus [-\lceil \frac{L}{4}\rceil,\frac{L}{4})^d} 
\cov{\ener_{\mu}(x)}{\ener_{\mu}(0)},\label{eq:def-I2}\\
I_3(L,\mu)&:=&-\sum_{x\in \big([-\lceil \frac{L}{2}\rceil,\frac{L}{2})^d\setminus [-\lceil \frac{L}{4}\rceil,\frac{L}{4})^d\big)\cap \Z^d} \cov{\ener_{L,\mu}(x)}{\ener_{L,\mu}(0)}.\label{eq:def-I3}
\end{eqnarray*}
The desired estimate \eqref{eq:resc-var-mu} then follows from the combination 
of the following three estimates:
\begin{eqnarray}
|I_1(L,\mu)|&\lesssim &L^d \log(2+\sqrt{\mu}L) e^{-c\sqrt{\mu}L},\label{eq:estim-I1}\\
|I_2(L,\mu)|&\lesssim &\log(2+\sqrt{\mu}L) e^{-c\sqrt{\mu}L},\label{eq:estim-I2}\\
|I_3(L,\mu)|&\lesssim &\log(2+\sqrt{\mu}L) e^{-c\sqrt{\mu}L},\label{eq:estim-I3}
\end{eqnarray}
which we prove in the last three steps.

\medskip

\step{2} Proof of \eqref{eq:estim-I1}.

\noindent Let $x\in [-\lceil \frac{L}{4}\rceil,\frac{L}{4})^d\cap \Z^d$.
By bilinearity of the covariance, the Cauchy-Schwarz and triangle inequalites, and stationarity of the energy densities
$\ener_\mu$ and $\ener_{L,\mu}$, we have
\begin{eqnarray}
\lefteqn{|\cov{\ener_{\mu}(x)}{\ener_{\mu}(0)}-\cov{\ener_{L,\mu}(x)}{\ener_{L,\mu}(0)}|}
\nonumber\\
&=&|\cov{\ener_{\mu}(x)-\ener_{L,\mu}(x)}{\ener_{\mu}(0)}
+\cov{\ener_{L,\mu}(x)}{\ener_{\mu}(0)-\ener_{L,\mu}(0)}|
\nonumber\\
&=&\Big|\E{\Big(\ener_{\mu}(x)-\ener_{L,\mu}(x)-\E{\ener_{\mu}(x)-\ener_{L,\mu}(x)}\Big) \Big(\ener_{\mu}(0)-\E{\ener_{\mu}(0)}\Big)}
\nonumber\\
&&+\E{\Big(\ener_{L,\mu}(0)-\ener_{\mu}(0)-\E{\ener_{\mu}(0)-\ener_{L,\mu}(0)}\Big) \Big(\ener_{L,\mu}(x)-\E{\ener_{L,\mu}(x)}\Big)}\Big|
\nonumber\\
&\leq &4\E{(\ener_{\mu}(x)-\ener_{L,\mu}(x))^2}^{\frac{1}{2}}\E{\ener_{\mu}^2}^{\frac{1}{2}}
+4\E{(\ener_{\mu}(0)-\ener_{L,\mu}(0))^2}^{\frac{1}{2}}\E{\ener_{L,\mu}^2}^{\frac{1}{2}}.
\label{eq:estim-I1-0}
\end{eqnarray}
On the one hand, by definition of the energy densities $\ener_{L,\mu}$ and $\ener_{\mu}$ and
by \cite[Proposition~1]{Gloria-Neukamm-Otto-14},
\begin{eqnarray}
\E{\ener_{L,\mu}^2}^{\frac{1}{2}}&\lesssim & \E{1+|\nabla \phi_{L,\mu}|^4}^{\frac{1}{2}} \,\lesssim\, 1 \label{eq:estim-I1-1}\\
\E{\ener_{\mu}^2}^{\frac{1}{2}}&\lesssim & \E{1+|\nabla \phi_{\mu}|^4}^{\frac{1}{2}} \,\lesssim\, 1.
\label{eq:estim-I1-2}
\end{eqnarray}
On the other hand, by definition of $\aa_L$, we have for all $z\in [-\lceil \frac{L}{2}\rceil,\frac{L}{2})^d\cap \Z^d$
\begin{eqnarray*}
\lefteqn{\ener_{\mu}(z)-\ener_{L,\mu}(z)}
\\
&=&(\xi+\nabla \phi_{\mu})\cdot\aa(\xi+\nabla \phi_{\mu})(z)-(\xi+\nabla \phi_{L,\mu})\cdot\aa_L(\xi+\nabla \phi_{L,\mu})(z) \\
&=&(\xi+\nabla \phi_{\mu})\cdot\aa(\xi+\nabla \phi_{\mu})(z)-(\xi+\nabla \phi_{L,\mu})\cdot\aa(\xi+\nabla \phi_{L,\mu})(z) \\
&=&(\nabla \phi_{\mu}-\nabla \phi_{L,\mu})\cdot\aa(\xi+\nabla \phi_{\mu})(z)
-(\xi+\nabla \phi_{L,\mu})\cdot\aa(\nabla \phi_{L,\mu}-\nabla \phi_{\mu})(z) ,
\end{eqnarray*}
so that by Cauchy-Schwarz' inequality and \cite[Proposition~1]{Gloria-Neukamm-Otto-14},
\begin{equation}
\E{(\ener_{\mu}(x)-\ener_{L,\mu}(z))^2}^\frac{1}{2}\,\lesssim \, \E{|\nabla \phi_{\mu}(z)-\nabla \phi_{L,\mu}(z)|^4}^\frac{1}{4}.\label{eq:estim-I1-3}
\end{equation}
It remains to estimate the RHS of \eqref{eq:estim-I1-3}.
To this end, recall that for $\mu>0$, the function $\phi_{L,\mu}$ is the unique bounded solution on $\Z^d$ of
$$
\mu \phi_{L,\mu}-\nabla^*\cdot \aa_L(\xi+\nabla \phi_{L,\mu})\,=\,0,
$$
where $\aa_L$ is the periodic extension on $\Z^d$ of $\aa|_{[-\lceil \frac{L}{2}\rceil,\frac{L}{2})^d\cap \Z^d}$.
Hence the difference $\delta_{L,\mu}:=\phi_{\mu}- \phi_{L,\mu}$ solves
$$
\mu \delta_{L,\mu}-\nabla^*\cdot\aa \nabla   \delta_{L,\mu}\,=\,\nabla^*\cdot (\aa-\aa_L)\nabla \phi_{L,\mu},
$$
and the Green representation formula yields for all $z\in \Z^d$,
\begin{eqnarray*}
\nabla \delta_{L,\mu}(z)&=&\sum_{y\in \Z^d} \nabla \nabla G_\mu(z,y)\cdot (\aa(y)-\aa_L(y))\nabla \phi_{L,\mu}(y)
\\
&=&\sum_{y\in \Z^d\setminus [-\lceil \frac{L}{2}\rceil,\frac{L}{2})^d} \nabla \nabla G_\mu(z,y)\cdot (\aa(y)-\aa_L(y))\nabla \phi_{L,\mu}(y).
\end{eqnarray*}
By the triangle inequality, Cauchy-Schwarz' inequality, 
the annealed estimate \eqref{eq:annealed-nabla2G} in Lemma~\ref{eq:ptwise-G}, and \cite[Proposition~1]{Gloria-Neukamm-Otto-14}, this turns into
\begin{eqnarray*}
\lefteqn{\E{|\nabla \phi_{\mu}(z)-\nabla \phi_{L,\mu}(z)|^4}^\frac{1}{4}}
\\
&\lesssim &
\sum_{y\in \Z^d\setminus [-\lceil \frac{L}{2}\rceil,\frac{L}{2})^d} \expec{|\nabla \nabla G_\mu(z,y)\cdot (\aa(y)-\aa_L(y))\nabla \phi_{L,\mu}(y)|^4}^{\frac{1}{4}}
\\
&\lesssim & \sum_{y\in \Z^d\setminus [-\lceil \frac{L}{2}\rceil,\frac{L}{2})^d} \expec{|\nabla \nabla G_\mu(z,y)|^8}^{\frac{1}{8}}\E{|\nabla \phi_{L,\mu}(y)|^8}^{\frac{1}{8}}
\\
&\lesssim& \sum_{y\in \Z^d\setminus [-\lceil \frac{L}{2}\rceil,\frac{L}{2})^d} \frac{e^{-c\sqrt{\mu}|z-y|}}{1+|z-y|^d}.
\end{eqnarray*}
Since $x\in [-\lceil \frac{L}{4}\rceil,\frac{L}{4})^d\cap \Z^d$ we thus have
\begin{equation}
\E{|\nabla \phi_{\mu}(x)-\nabla \phi_{L,\mu}(x)|^4}^\frac{1}{4}
\,\lesssim \, \sum_{y\in \Z^d\setminus [-\lceil \frac{L}{4}\rceil,\frac{L}{4})^d} \frac{e^{-c\sqrt{\mu}|y|}}{1+|y|^d}\,\lesssim\, \log(2+\sqrt{\mu}L) e^{-c\sqrt{\mu}L}.\label{eq:estim-I1-4}
\end{equation}
The combination of \eqref{eq:estim-I1-0}---\eqref{eq:estim-I1-4} proves the desired estimate \eqref{eq:estim-I1} by summation over $x\in [-\lceil \frac{L}{4}\rceil,\frac{L}{4})^d\cap \Z^d$.

\medskip

\step{3} Proof of \eqref{eq:estim-I2}.

\noindent This is a direct consequence of the covariance estimate \eqref{eq:covar-x-0-ener} in the proof of Lemma~\ref{lem:conv-series}, after summation over $\Z^d \setminus  [-\lceil \frac{L}{4}\rceil,\frac{L}{4})^d$:
\begin{multline*}
|I_2(L,\mu)|\,\leq\,\sum_{x\in \Z^d \setminus  [-\lceil \frac{L}{4}\rceil,\frac{L}{4})^d} |\cov{\ener_\mu(x)}{\ener_\mu(0)}| \\
\stackrel{\eqref{eq:covar-x-0-ener}}{\lesssim} \,\sum_{x\in \Z^d \setminus  [-\lceil \frac{L}{4}\rceil,\frac{L}{4})^d} \frac{e^{-c\sqrt{\mu}|x|}}{1+|x|^d} \,\lesssim\, \log(2+\sqrt{\mu}L) e^{-c\sqrt{\mu}L}.
\end{multline*}

\medskip

\step{4} Proof of \eqref{eq:estim-I3}.

\noindent The proof is similar to the proof of \eqref{eq:estim-I2}.
The starting point is the defining equation \eqref{eq:prop7-1.0} on the torus.
Using the covariance estimate of Lemma~\ref{lem:covar} on $\Omega_L$, the same
proof as for \eqref{eq:covar-x-0-ener.0} leads to
\begin{equation*}
\EL{\sup_{\aa_L(e)}\Big|\frac{\partial \ener_{L,\mu}(x)}{\partial \aa_L(e)}\Big|^2}^{\frac{1}{2}}
\,\lesssim\, \delta(z-x)+\frac{e^{-c\sqrt{\mu}|z-x|}}{1+|z-x|^d},
\end{equation*}
for all $x,z\in \T_L$, $e=(z,z+\ee_i)$ for some $i\in \{1,\dots,d\}$, and where (in this step) $|x-z|$ denotes the  distance on the torus $\T_L$.
Using the identity $\cov{\ener_{L,\mu}(x)}{\ener_{L,\mu}(0)}=\covL{\ener_{L,\mu}(x)}{\ener_{L,\mu}(0)}$, this yields as in Step~3,
\begin{multline*}
|I_3(L,\mu)|
\,\leq\, \sum_{x\in [-\lceil \frac{L}{2}\rceil,\frac{L}{2})^d \setminus  [-\lceil \frac{L}{4}\rceil,\frac{L}{4})^d\cap \Z^d} |\covL{\ener_{L,\mu}(x)}{\ener_{L,\mu}(0)}| 
\\
\,\lesssim \,\sum_{x\in  [-\lceil \frac{L}{2}\rceil,\frac{L}{2})^d \setminus [-\lceil \frac{L}{4}\rceil,\frac{L}{4})^d\cap \Z^d} \frac{e^{-c\sqrt{\mu}|x|}}{1+|x|^d} \,\lesssim\, \log(2+\sqrt{\mu}L) e^{-c\sqrt{\mu}L}.
\end{multline*}
\end{proof}
%

%% file: normal.tex
\subsection{Structure of the proof and auxiliary results}

Using a version of Stein's method, developed by Chatterjee \cite{Chatterjee-08} and extended by Lachi\`eze-Rey and Peccati \cite{LRP-15}, we shall prove the following: 
\begin{prop}\label{prop:normal}
Let $\sigma_L^2:=L^d \var{A_L}$. There exists $r>8$ such that
\begin{equation}
d_{K} \Big(L^{\frac{d}{2}} \frac{A
_L- \expE[A_L]}{ \sigma_L},\mathcal G \Big)\,\lesssim \,  \frac{L^{-\frac{d}{2}}}{\sigma_L^3}(1 + \expec{|\nabla \phi_{L}|^{12}}^\frac{1}{2})+\frac{L^{-\frac{d}{2}}\log L}{\sigma_L^2}(1 + \expec{|\nabla \phi_{L}|^{r}}^{\frac{4}{r}}),
\end{equation}
where $\mathcal G$ is a standard normal random variable.
\qed
\end{prop}
The combination of Proposition~\ref{prop:normal} with the moments bounds $\expec{|\nabla \phi_{L}|^q}\,\lesssim\,1$ for all $q\ge 1$
from \cite[Proposition~1]{Gloria-Neukamm-Otto-14} and with the lower bound on $\sigma_L$ from 
 Proposition \ref{prop:lower-bound} proves Proposition~\ref{th:normal}.

\medskip

The structure of the argument for Proposition~\ref{prop:normal} is similar to the analysis in \cite{Nolen-14} in the continuum setting. The main ingredient is a result on normal approximation \cite{Chatterjee-08, LRP-15}.
Under the product measure $\Pm_L$ on $\Omega_L = [\lambda,1]^{\mathbb{B}_L}$, the conductances $\{a(j)\}_{j \in \mathbb{B}_L}$  are independent random variables. Let $f \in L^2(\Omega_L,\Pm_L)$; we want to know whether the distribution of $f(a)$ is close to normal.

Let $a' = \{a'(j)\}_{j \in \mathbb{B}_L}$ denote an independent copy of the random conductance $a = \{a(j)\}_{j \in \mathbb{B}_L}$. Given $j \in \mathbb{B}_L$, define a new coefficient $a^j = \{a^j(\ell) \}_{\ell \in B_L}$ by
\begin{equation}
a^{j}(\ell) = \left\{ \begin{array}{cc} a(\ell), & \quad \ell \neq j \\ a'(\ell), & \quad \ell = j,  \end{array} \right. \label{ajdef}
\end{equation}
for $\ell \in \mathbb{B}_L$. That is, we replace the $j^{th}$ component of $a$ by the $j^{th}$ component of $a'$.  Similarly, for any set of edges $B \subset  \mathbb{B}_L$, we define a coefficient $a^B$ by replacing $a(\ell)$ by $a'(\ell)$, for all $\ell \in B$. Next, define the discrete difference
$$
\Delta_j f(a) = f(a^j) - f(a).
$$
This is a function of both $a$ and $a'$ (of $a$ and $a'(j)$) and we sometimes write $\Delta_j f(a,a')$ to emphasize this point. If $j \notin B$, then define
$$
\Delta_j f(a^B) = f(a^{B \cup \{j\}}) - f(a^B).
$$
The following identity is due to Lachi\`eze-Rey and Peccati: 
\begin{theo} \cite[Theorem 4.2]{LRP-15} \label{theo:normalapprox}
Let $f \in L^2(\Omega_L,\Pm_L)$ satisfy $\expE[f] = \mu$ and $\var{f} = s^2$. Then 
\begin{multline}
d_{K} \left( \frac{f - \mu}{s},Y\right) \leq  \frac{\sqrt{2\pi}}{16 s^3} \sum_{j \in \mathbb{B}_L} \expE \left [ |\Delta_j f(a)|^3 \right ]  
+ \frac{1}{4 s^3} \sum_{j \in \mathbb{B}_L} \expE \left [ |\Delta_j f(a)|^6 \right ]^\frac{1}{2}
\\
+ \frac{1}{s^2} \var{\expE[ T(a,a') | a]}^{\frac{1}{2}}+\frac{1}{s^2} \var{\expE[ T'(a,a') | a]}^{\frac{1}{2}}  \label{normalapproxsum}
\end{multline}
where $Y \sim N(0,1)$,
\begin{eqnarray*}
T(a,a')& =& \frac{1}{2} \sum_{j \in \mathbb{B}_L} \sum_{\substack{B \subset \mathbb{B}_L \\ j \notin B  } } K_{L,B}  \Delta_j f(a) \Delta_j f(a^{B}),\\
T'(a,a')& =& \frac{1}{2} \sum_{j \in \mathbb{B}_L} \sum_{\substack{B \subset \mathbb{B}_L \\ j \notin B  } } K_{L,B}  \Delta_j f(a) |\Delta_j f(a^{B})|,
\end{eqnarray*}
and $K_{L,B} = |B|! (|\mathbb{B}_L| - |B| - 1)! /(|\mathbb{B}_L|!)$.
\end{theo}

\begin{rem} The right-hand side of \eqref{normalapproxsum} controls
not only the Kolmogorov distance, but also the Wasserstein distance (up to a factor $2$, cf. \cite[Theorem 2.2]{Chatterjee-08}).
If $a$ is a smooth function of Gaussian variables, then one can also control the total variation, cf. \cite[Theorem 3.1]{Chatterjee-14}.
\end{rem}

\begin{rem} \label{rem:avgSets}
If $S_{L,j}$ denotes the collection of all subsets $B \subset \mathbb{B}_L$ which do not contain the index $j$, then the weights $K_{L,B} \geq 0$ define a probability measure on $S_{L,j}$: $\sum_{B \in S_{L,j}} K_{L,B}=1$. 
\end{rem}

\subsection{Proof of Proposition~\ref{prop:normal}}

We split the proof into four steps.
The strategy consists in applying Theorem~\ref{theo:normalapprox} to $A_L$ and 
use the moment bounds on $\nabla \phi_L$ of \cite[Proposition~1]{Gloria-Neukamm-Otto-14}
and the annealed bounds on the Green function of \cite{Marahrens-Otto-13} gathered in Lemma~\ref{lem:Green-annealed} to estimate the RHS of \eqref{normalapproxsum}.
In the first step we simplify the RHS of \eqref{normalapproxsum} using the Efron-Stein inequality.
In the second step we control the sensitivity $\Delta A_L$ of the effective conductance with respect to 
local changes of the conductances via the sensitivity $\Delta \phi_L$ of the corrector itself.
The sensitivity of the corrector is then estimated in Step~3, and we conclude the proof in Step~4.

\medskip

\step{1} Reformulation of the RHS of \eqref{normalapproxsum}.

\noindent Let us define
\begin{eqnarray*}
h_j(a) &=&  \frac{1}{2} \sum_{\substack{B \subset \mathbb{B}_L\\ j \notin B  } } K_{L,B} \expE[ \Delta_j f(a) \Delta_j f(a^{B}) \;\big|\; a],
\\
h_j'(a) &=&  \frac{1}{2} \sum_{\substack{B \subset \mathbb{B}_L\\ j \notin B  } } K_{L,B} \expE[ \Delta_j f(a) |\Delta_j f(a^{B})| \;\big|\; a],
\end{eqnarray*}
so that $\expE[ T(a,a') | a] = \sum_{j} h_j(a)$ and $\expE[ T'(a,a') | a] = \sum_{j} h_j'(a)$. The terms $\var{\expE[ T(a,a') | a]}$ and $\var{\expE[ T'(a,a') | a]}$
in the RHS of \eqref{normalapproxsum}  can be estimated by the Efron-Stein inequality (\cite{Efron-Stein}, \cite{Steele86}). To this end, we introduce a third coefficient $a'' = \{a''(\ell)\}_{\ell \in \mathbb{B}_L} $ which is an independent copy of $a$ and $a'$. Let us define coefficient $a^k$ by
\begin{equation}
a^{k}(\ell) = \left\{ \begin{array}{cc} a(\ell), & \quad \ell \neq k \\ a''(\ell), & \quad \ell = k  \end{array} \right. \label{akdef}
\end{equation}
for all $\ell \in \mathbb{B}_L$. For any function $g(a,a'): \Omega_L \times \Omega_L \to \Rm$ we define
\begin{equation}
\Delta_k g(a,a') = g(a^k,a') - g(a,a').
\end{equation}
In particular, $\Delta_k g(a,a') = 0$ if $g(a,a')$ does not depend on $a_k$. We use the notation $g^k$ to denote the action of replacing $a(k)$ by $a''(k)$ in the argument of $g$:
$$
g(a,a')^k = g(a^k,a').
$$
Thus, $\Delta_k g(a,a') = (g(a,a'))^k - g(a,a')$. Let us emphasize that $a^k$ will always refer to \eqref{akdef} while $a^j$ refers to \eqref{ajdef}. The coefficients $a^j$ and $a^k$ have the same law, but the coefficient denoted by $a^j$ is not equivalent to $a^k$ even when the values of the indices $k$ and $j$ are the same.

Applying the Efron-Stein inequality and then Minkowski's inequality, we have
\begin{eqnarray}
\var{\expE[ T(a,a') | a]} =  \var{\sum_{j \in \mathbb{B}_L} h_j(a) } & \leq & \frac{1}{2} \sum_{k \in \mathbb{B}_L}   \expE\left[ \Big|  \sum_{ j \in \mathbb{B}_L}  \Delta_k h_j(a) \Big|^2\right] \no \\
& \leq & \frac{1}{2} \sum_{k \in \mathbb{B}_L} \left( \sum_{j \in \mathbb{B}_L} \expE[|\Delta_k h_j(a) |^2 ]^{\frac{1}{2}} \right)^2, \label{efstein1}
\end{eqnarray}
and likewise for $\var{\expE[ T'(a,a') | a]}$.
By Jensen's inequality,
\begin{eqnarray}
 \expE[|\Delta_k h_j(a) |^2 ]^{\frac{1}{2}} &\leq&  \expE[(\Delta_k \Delta_j f(a) )^2\overline{\Delta_j f(a^B)} ^2 ]^{\frac{1}{2}} + \expE[(\Delta_j f(a^k))^2( \Delta_k \overline{\Delta_j f(a^B)} )^2 ]^{\frac{1}{2}}, \label{hjensen}\\
 \expE[|\Delta_k h'_j(a) |^2 ]^{\frac{1}{2}}& \leq & \expE[|\Delta_k \Delta_j f(a) \overline{|\Delta_j f(a^B)|} |^2 ]^{\frac{1}{2}} + \expE[(\Delta_j f(a^k))^2 (\Delta_k \overline{|\Delta_j f(a^B)|} )^2 ]^{\frac{1}{2}}, \label{hjensen2}
\end{eqnarray}
where we have used the notation $\overline{\cdot}$ to indicate averaging with respect to the set $B$. Specifically, if $S_{L,j}$ denotes the collection of all subsets $B \subset \mathbb{B}_L$ which do not contain the index $j$  (recall Remark \ref{rem:avgSets}), then 
\begin{eqnarray}
 \overline{\Delta_j f(a^B)} & = &\sum_{\substack{B \subset  \mathbb{B}_L\\ j \notin B  } } K_{L,B} \Delta_j f(a^B) = \sum_{B \in S_{L,j}} K_{L,B} \Delta_j f(a^B), \label{Aavg}\\
\overline{|\Delta_j f(a^B)|} & = &\sum_{\substack{B \subset  \mathbb{B}_L\\ j \notin B  } } K_{L,B} |\Delta_j f(a^B)| = \sum_{B \in S_{L,j}} K_{L,B} |\Delta_j f(a^B)|. \label{Aavg2}
\end{eqnarray}

From \eqref{efstein1}, \eqref{hjensen}, \eqref{hjensen2}, and the triangle inequality in the form $|\overline{\Delta_j f(a^B)}| \leq \overline{|\Delta_j f(a^B)|} $,
we see that the right side of \eqref{normalapproxsum} is controlled by the sums
\begin{eqnarray}
S_1 &=& \sum_{k \in \mathbb{B}_L} \left( \sum_{j \in \mathbb{B}_L} \expE[(\Delta_k \Delta_j f(a))^2 \overline{|\Delta_j f(a^B)|}^2 ]^{\frac{1}{2}} \right)^2 , \label{s1def}
\\
S_2 &=& \sum_{ k \in \mathbb{B}_L} \left( \sum_{j \in \mathbb{B}_L} \expE[(\Delta_j f(a^k))^{2} (\Delta_k \overline{|\Delta_j f(a^B)|})^2 ]^{\frac{1}{2}} \right)^2, \label{s2def}
\\
S_3 &=& \sum_{j \in \mathbb{B}_L} \expE[ |\Delta_j f(a)|^6 ]^\frac{1}{2}. \label{s3def}
\end{eqnarray}

\medskip

\step{2} Sensitivity estimate of $f$ and control of $\Delta f_j$ and $\Delta_k \Delta_j f$.

\noindent As announced, we apply Theorem~\ref{theo:normalapprox} to the random variable
$$
f(a) = A_L = L^{-d} \sum_{x \in \T_L} (\xi + \nabla \phi_{L}) \cdot \aa (\xi + \nabla \phi_{L})(x).
$$
We thus need to compute and estimate the terms $\Delta_j f$ and $\Delta_k \Delta_j f$ appearing in \eqref{normalapproxsum}. 
Let us introduce the notation $\phi_{L}^j = \phi_{L}(x,a^j)$, $\phi_{L}^k = \phi_{L}(x,a^k)$ (recall \eqref{ajdef} and \eqref{akdef}). Given an edge $j \in \mathbb{B}_L$ such that $j = (x,x + \ee_\ell)$, we use the notation $\nabla \phi_{L}(j)$ for the scalar $[\nabla \phi_{L}(x)]_\ell = \phi_{L}(x + \ee_\ell) - \phi_{L}(x)$, which is the forward difference of $\phi_{L}$ across the edge $j$.
We claim that for all $j \neq k \in \B_L$ and $q\ge 1$,
\begin{eqnarray}
L^{qd} \expE[ |\Delta_j f(a)|^q] &\lesssim& 1 + \expE[ |\nabla \phi_{L}(j,a)|^{2q}], \label{deljfmoment}
\\
L^{qd} \expE[ \overline{|\Delta_j f(a^B)|}^q] &\lesssim& 1 + \expE[ |\nabla \phi_{L}(j,a)|^{2q}], \label{GammaBarBound}
\\
L^{2d} (\Delta_{k} |\Delta_{j} f(a)|)^2 & \lesssim & \Big(1 +  |\nabla \phi_{L}(j,a)|^2 + |\nabla \phi_{L}(j,a^k)|^2 + |\nabla \phi_{L}(j,a^j)|^2 \Big) \no \\
& & \times \left(|\nabla \Delta_k \phi_{L}(j,a)|^2 + |\nabla \Delta_k \phi_{L}(j,a^j)|^2 \right). \label{deljf-ant1}
\end{eqnarray}
By \eqref{eq:antoine-shorter} in the proof of Proposition~\ref{prop:lower-bound}, we have
$$
L^d \Delta_j f(a) \,=\, L^d (A_L^j-A_L) \,=\, \sum_{x \in \T_L} (\nabla \phi_{L}^j +  \xi) \cdot (\Delta_j \aa) (\nabla \phi_{L} +  \xi)(x). 
$$
We will say that the matrix $\Delta_j \aa = \text{diag}[ a^j(x,x + \ee_1) - a(x,x + \ee_1)\, , \, \dots \,,\, a^j(x,x+ \ee_d) - a(x,x + \ee_d)]$ is ``supported on edge $j$" to mean that if $j = (y,y + \ee_\ell)$ then $(\Delta_j \aa(x))_{m,n} = 0$ except when $x = y$ and $m = n = \ell$.  Hence
\br
L^d \Delta_j f(a) & = &  \nabla (\phi_{L}^j + x \cdot \xi)(j) (a'(j) - a(j)) \nabla( \phi_{L} +  x \cdot \xi)(j),  \label{Deljf}
\er
which immediately implies 
%
$$
L^d |\Delta_j f(a)| \,\lesssim \,1 + |\nabla \phi_{L}(j,a)|^2 + |\nabla \phi_{L}(j,a^j)|^2.
$$
This estimate yields  \eqref{deljfmoment}. Indeed, since $\phi_{L}$ is stationary with respect to integer shifts and because $a$ and $a^j$ have the same law, the random variables $|\nabla \phi_{L}(j,a)|$ and $|\nabla \phi_{L}(j,a^j)|$ are identically distributed, so that \eqref{deljfmoment} follows. 

Now we prove \eqref{GammaBarBound}. Jensen's inequality implies
$$
\overline{|\Delta_j f(a^B)|}^q = \big(\sum_{\substack{B \subset  \mathbb{B}_L \\ j \notin B }} K_{L,B}  |\Delta_j f(a^B)|\big)^q \leq \sum_{\substack{B \subset  \mathbb{B}_L\\ j \notin B }} K_{L,B}  |\Delta_j f(a^B)|^q .
$$
Therefore from \eqref{deljfmoment} we obtain
$$
L^{qd} \expE[\overline{|\Delta_j f(a^B)|}^q] \leq L^{qd} \sum_{\substack{B \subset  \mathbb{B}_L \\ j \notin B }} K_{L,B} \expE[ |\Delta_j f(a^B)|^q] \leq 1 + C_q \expE[|\nabla \phi_{L}(0,a)|^{2q}].
$$

It remains to prove \eqref{deljf-ant1}.  If $j \neq k$, then $(a^j)^k=(a^k)^j$, and therefore we have by the triangle inequality
\begin{multline*}\label{eq:antoine-simpler2}
\big|\Delta_k |\Delta_j f(a)|\big| \,\leq \, |\Delta_k \Delta_j f(a)|\,=\,|\Delta_k f(a^j)-\Delta_k f(a)| \,=\, 
|f((a^j)^k)-f(a^j)-(f(a^k)-f(a))|
\\
\,=\, |f((a^k)^j)-f(a^k)-(f(a^j)-f(a)) |=\,|\Delta_j f(a^k)-\Delta_j f(a)|.
\end{multline*}
We then insert \eqref{Deljf} and rearrange the terms:
\begin{eqnarray*}
&&\big|L^d\Delta_k |\Delta_j f(a)|\big|
\\ &\leq&|L^d (\Delta_j f(a^k)-\Delta_j f(a))|
 \\
&=&|\nabla (\phi_L^{jk}+x\cdot \xi)(j)(a'(j)-a(j))\nabla (\phi_L^k+x\cdot \xi)(j)
\\
&&\qquad-\nabla(\phi_L^j+x\cdot \xi)(j)(a'(j)-a(j))\nabla (\phi_L+x\cdot \xi)(j)|
\\
&=&|\nabla (\phi_L^{jk}-\phi_L^j)(j)(a'(j)-a(j))\nabla (\phi_L^k+x\cdot \xi)(j)
\\
&&\qquad -\nabla (\phi_L^j+x\cdot \xi)(j)(a'(j)-a(j)) \nabla (\phi_L-\phi_L^k)(j) |\\
&=&|\nabla \Delta_k \phi_L^j (j)(a'(j)-a(j))\nabla (\phi_L^k+x\cdot \xi)(j) 
-\nabla \Delta_k \phi_L(j) (a'(j)-a(j)) \nabla (\phi_L^j+x\cdot \xi)(j)|.
\end{eqnarray*}
This turns into \eqref{deljf-ant1} by Cauchy-Schwarz' inequality. If $k = j$, then $(a^j)^k = a^j$ so that $\big|\Delta_k |\Delta_j f(a)|\big| \leq |\Delta_k\Delta_j f(a)|=|\Delta_kf(a)|$. Hence we have 
\begin{equation}
L^{2d} (\Delta_{k} |\Delta_{j} f(a)|)^2 \leq L^{2d} |\Delta_{k} f(a)|^2  \lesssim \,1 + |\nabla \phi_{L}(j,a)|^4 + |\nabla \phi_{L}(j,a^k)|^4 \label{jeqkD2}
\end{equation}
when $j = k$.

\medskip

\step{3} Sensitivity estimate of $\nabla \phi_{L}$.
We claim that for $w_k := \Delta_k \phi_{L}(x,a)$, we have
\begin{equation}
|\nabla w_k(j)| \leq |\nabla_x \nabla_y G_L(j,k,a)| \left(|\nabla \phi_{L}(k,a^k)| + 1\right), \label{wkbound}
\end{equation}
where $\nabla_x \nabla_y G_L(j,k,a)$ is the mixed second gradient of the periodic Green function 
$G_L(x,y,a)$ of Lemma \ref{lem:Green-annealed}  (with $\mu = 0$) at edges $j= (x,x + \ee_\ell)$ and $k = (y,y + \ee_m)$.

Indeed, since the function $w_k$ satisfies
$$
- \nabla^* \cdot \aa_L \nabla w_k = \nabla^* \cdot (\Delta_k \aa_L)(\nabla \phi_{L}^k + \xi) \quad \text{in} \;\; \T_L,
$$
the Green representation formula yields
\br
w_k(x) & = & \sum_{y \in \T_L} G_L(x,y,a) \nabla^* \cdot (\Delta_k \aa_L)(\nabla \phi_{L}^k + \xi)(y) \no  \\
& = & - \sum_{y \in \T_L} \nabla_y G_L(x,y,a) \cdot (\Delta_k \aa_L)(\nabla \phi_{L}^k + \xi)(y).
\er
Hence,
\br
\nabla w_k(x) = - \sum_{y \in \T_L} \nabla_x \nabla_y G_L(x,y,a) \cdot (\Delta_k \aa_L)(\nabla \phi_{L}^k + \xi)(y),
\er
from which \eqref{wkbound} follows since $\Delta_k \aa$ is supported on edge $k$, and $|\Delta_j \aa| \leq 1$. 


\medskip

\step{4} Conclusion.

\noindent It remains to combine the estimates of Steps~2 and~3 to bound the sums $S_1$, $S_2$, and $S_3$ defined in \eqref{s1def}--\eqref{s3def}. 
We start with $S_1$.
By \eqref{deljf-ant1} in Step~2 and \eqref{wkbound} in Step~3, we have
\br
  |\Delta_k \Delta_j f(a)|^2 & \lesssim & L^{-2d}  \left(1 + |\nabla \phi_L(j,a)|^2 + |\nabla \phi_L(j,a^j)|^2 + |\nabla \phi_L(j,a^k)|^2 \right) \no \\
& & \times \left( (1 + |\nabla \phi_L(k, a^k)|^2)(\nabla_x \nabla_y G_L(j,k,a))^2  + \right. \no \\
&& \quad \quad \left. + (1 + |\nabla \phi_L(k, a^{jk})|^2)(\nabla_x \nabla_y G_L(j,k,a^j))^2\right). 
\er
for all $j \neq k$. All of the $|\nabla \phi_L|^2$ terms here have the same marginal distribution. Combined with 
\eqref{GammaBarBound}, this yields by H\"older's inequality with exponents $p > 1$ and $q = 4\frac{p}{p-1}$ (so that $\frac{2}{q} + \frac{1}{q} + \frac{1}{q} + \frac{1}{p} = 1$):
\begin{equation}
\expE[\overline{|\Delta_j f(a^B)|}^2 (\Delta_k \Delta_j f(a)  )^2 ]^{\frac{1}{2}} \, \lesssim \,  L^{-2d} \left(1 + \expE[|\nabla \phi_{L}|^{2q}]^{\frac{2}{q}}  \right) \expE[ |\nabla_x \nabla_y G_L(k,j)|^{2p}]^{\frac{1}{2p}}. \label{S1bound}
\end{equation}
By the annealed estimate \eqref{eq:annealed-nabla2G} (for $\mu=0$) in Lemma~\ref{lem:Green-annealed},
this turns into
\br
\expE[\overline{|\Delta_j f(a^B)|}^2(\Delta_k \Delta_j f(a) )^2 ]^{\frac{1}{2}} & \lesssim & L^{-2d}  \left(1 + \expE[|\nabla \phi_{L}|^{2q}]^{\frac{2}{q}}  \right) (1 + |j - k|)^{-d}.
\er
The same estimate holds in the case $j = k$; this is obtained by using (\ref{jeqkD2}) in the place of (\ref{deljf-ant1}) to bound $\Delta_k \Delta_j f$. Therefore, we have proved
$$
S_1 \lesssim L^{-3d} (\log(L))^2 \left(1 + \expE[|\nabla \phi_{L}|^{2q}]^{\frac{4}{q}}  \right).
$$

The estimate of the sum $S_2$ is very similar.  By Jensen's inequality (recall Remark \ref{rem:avgSets}), and the triangle inequality,
\begin{equation*}
\left(  \Delta_k (\overline{|\Delta_j f(a^B)|}) \right)^2 \,\leq \,\sum_{B \in S_{L,j}}  K_{L,B}  \left(  \Delta_k (|\Delta_j f(a^B)|) \right)^2
\,
\leq\,  \sum_{B \in S_{L,j}}  K_{L,B}  \left(  \Delta_k (\Delta_j f(a^B)) \right)^2.
\end{equation*}
If $k \in B \subset \B_L$, then $\Delta_k (\Delta_j f(a^B)) = 0$, since $\Delta_j f(a^B)$ doesn't depend on $a(k)$ in this case (although it does depend on $a'(k)$). On the other hand, if $k \notin B \cup \{j\}$, then $\Delta_k (\Delta_j f(a^B)) =  \Delta_k \Delta_j f(a^B)$ has the same as the law of $\Delta_k \Delta_j f(a)$. Consequently,
by \eqref{deljfmoment} and \eqref{deljf-ant1} in Step~2, \eqref{wkbound} in Step~3, and H\"older's inequality as above, we obtain
%
\br
& & \expE[(\Delta_j f(a^k))^{2} (\Delta_k \overline{|\Delta_j f(a^B)|} )^2 ] \no \\
& & \quad \quad \quad \quad \leq \sum_{B \in S_{L,j}}  K_{L,B}  \expE \left[ |(\Delta_j f(a^k)))^{2}  (\Delta_k \Delta_j f(a^B) )^2 \right] \no \\
& & \quad \quad \quad \quad \lesssim L^{-4d} \sum_{B \in S_{L,j}}  K_{L,B}  \left(1 + \expE[|\nabla \phi_{L}(j,a)|^{2q}]^{\frac{4}{q}}  \right)\expE[ |\nabla_x \nabla_y G_L(k,j,a)|^{2p}]^{\frac{1}{p}} \no \\
& & \quad \quad \quad \quad = L^{-4d} \left(1 + \expE[|\nabla \phi_{L}|^{2q}]^{\frac{4}{q}}  \right)\expE[ |\nabla_x \nabla_y G_L(k,j)|^{2p}]^{\frac{1}{p}}.  \no
\er
Therefore, $\expE[(\Delta_j f(a^k))^{2} (\Delta_k \overline{|\Delta_j f(a^B)|})^2 ]^{\frac{1}{2}} \lesssim  L^{-2d}  \left(1 + \expE[|\nabla \phi_{L}|^{2q}]^{\frac{2}{q}}  \right) (1 + |j - k|)^{-d}$, as above, which implies
$$
S_2 \lesssim L^{-3d} (\log(L))^2 \left(1 + \expE[|\nabla \phi_{L}|^{2q}]^{\frac{4}{q}}  \right).
$$

By \eqref{deljfmoment} we also have
$$
S_3 \lesssim L^{-2d} \left(1 + \expE[|\nabla \phi_{L}|^{12}]^\frac{1}{2} \right).
$$
In view of Theorem \ref{theo:normalapprox}, \eqref{efstein1}, \eqref{hjensen}, and these bounds on $S_1$, $S_2$, and $S_3$ we conclude that
$$
d_{K} \left( \frac{f - \expE[f]}{L^{-\frac{d}{2}} \sigma_L },Y\right) \lesssim \frac{L^{-\frac{d}{2}}}{\sigma_L^3}  \left(1 + \expE[|\nabla \phi_{L}|^{12}] ^{\frac{1}{2}} \right) + \frac{L^{-\frac{d}{2}} \log L}{\sigma_L^2} \left(1 + \expE[|\nabla \phi_{L}|^{2q}]^{\frac{2}{q}}  \right),
$$
as desired.